\documentclass[12pt]{amsart}

\usepackage[margin = 1in]{geometry}

\usepackage{tikz}
\usepackage{tikz-cd}
\usepackage{url, hyperref}
\usepackage{float}

\tikzset{->-/.style={decoration={
  markings,
  mark=at position .45 with {\arrow{>}}},postaction={decorate}}}
  
\usetikzlibrary{shapes.arrows}
\usetikzlibrary{decorations.pathreplacing}
\usepackage[all]{xy}

\tikzset{->-/.style={decoration={
  markings,
  mark=at position .45 with {\arrow{>}}},postaction={decorate}}}

\usepackage{amsmath}
\usepackage{amssymb}
\usepackage{enumitem}
\usepackage{graphicx}
\usepackage{mathdots}
\usepackage{color}
\usepackage{diagbox}
\usepackage{array, makecell}
\usepackage{rotating}
\usepackage{amsmath}
\usepackage{tikz-cd}

\newcount\colveccount
\newcommand*\colvec[1]{
        \global\colveccount#1
        \begin{pmatrix}
        \colvecnext
}
\def\colvecnext#1{
        #1
        \global\advance\colveccount-1
        \ifnum\colveccount>0
                \\
                \expandafter\colvecnext
        \else
                \end{pmatrix}
        \fi
}

\usepackage{amsthm}

\theoremstyle{definition}
\newtheorem{definition}{Definition}[subsection]
\newtheorem{theorem}[definition]{Theorem}
\newtheorem{example}[definition]{Example}
\newtheorem{proposition}[definition]{Proposition}
\newtheorem{situation}[definition]{Situation}

\newtheorem{corollary}[definition]{Corollary}
\newtheorem{lemma}[definition]{Lemma}

\newtheorem{remark}[definition]{Remark}

\newtheorem{notation}[definition]{Notation}

\def\bK{\mathbb{K}}

\def\bP{\mathbb{P}}

\def\cF{\mathcal{F}}

\def\cL{\mathcal{L}}
\def\cM{\mathcal{M}}
\def\cN{\mathcal{N}}
\def\cO{\mathcal{O}}

\def\barM{\overline{\mathcal{M}}}

\DeclareMathOperator{\Aut}{Aut}
\DeclareMathOperator{\Bl}{Bl}
\DeclareMathOperator{\charac}{char}

\DeclareMathOperator{\ev}{ev}
\def\Ex{\mathrm{Ex}}

\DeclareMathOperator{\Gr}{Gr}

\DeclareMathOperator{\Pic}{Pic}

\DeclareMathOperator{\Supp}{Supp}

\newcommand{\uw}{\underline{W}}
\newcommand{\dtilde}{\widetilde{d}}
\newcommand{\rtilde}{\widetilde{r}}
\newcommand{\ws}{\widetilde{s}}

\title{Brill-Noether existence with secants}

\author{Alessio Cela}
\address{ University of Cambridge, Department of pure mathematics and mathematical statistics
\hfill \newline\texttt{}
 \indent Centre for Mathematical Sciences, Wilberforce Road Cambridge, UK} \email{{\tt ac2758@cam.ac.uk}}

 \author{Carl Lian}
\address{Washington University in St. Louis, Department of Mathematics, 1 Brookings Drive
\hfill \newline\texttt{}
 \indent  St. Louis, MO 63130} \email{{\tt clian@wustl.edu}}

\date{\today}

\usepackage{graphicx}

\begin{document}

\maketitle

\begin{abstract}
We completely characterize when a general curve of genus $g$ admits a non-degenerate degree $d$ map to projective space $\bP^r$, that is $k$-secant along a linear space $\bP^s \subset \bP^r$ and deforms in a smooth family of expected dimension. This gives an optimal improvement of a theorem of Farkas, who gave necessary conditions, and proved sufficiency under additional numerical assumptions. Our main theorem shows that the natural necessary conditions are in fact sufficient, yielding a generalization of the classical Brill-Noether existence theorem. The proof relies on deformation theory of stable maps, and is valid in arbitrary characteristic.
\end{abstract}

\setcounter{tocdepth}{1}

\section{Introduction}\label{sec:intro}

Let $C$ be a smooth, proper, connected curve of genus $g$ defined over an algebraically closed field $\bK$. Then, $C$ is projective. However, $C$ can be realized as a projective subvariety of $\bP^r$, or more generally mapped to $\bP^r$, in many ways. It is natural to ask, for which values of $r,d$, there exists a map $f\colon C \to \bP^r$ of degree $d$. When $C$ is general in moduli, a complete answer is given by the classical Brill-Noether existence theorem \cite{gh}.

\begin{theorem}\label{thm:BN_existence}
Let $C$ be a general curve of genus $g$, by which we mean a general point $C \in \cM_g$. Then, there exists a non-degenerate map $f\colon C \to \bP^r$ of degree $d$ if and only if the \emph{Brill-Noether number}
\[
\rho(g,r,d) := g - (r+1)(g-d+r)
\]
is non-negative.
\end{theorem}

A map $f$ is \emph{non-degenerate} if its image is not contained in any hyperplane. Restricting to non-degenerate maps is natural, because a degenerate map may instead be regarded as a map to a lower-dimensional projective space. The Gieseker-Petri theorem \cite{gieseker,eh_petri} asserts furthermore that the moduli space of non-degenerate maps $f\colon C \to \bP^r$ is smooth, of dimension
\[
h^0(C,f^*T_{\bP^r}) = \chi(C,f^*T_{\bP^r}) = \rho(g,r,d)+((r+1)^2-1).
\]
Equivalently, $H^1(C,f^*T_{\bP^r})=0$. 

The quantity $(r+1)^2-1$ is the dimension of the automorphism group of $\bP^r$; note that any non-degenerate map can be post-composed with any such automorphism, so necessarily moves in a family of dimension at least $(r+1)^2-1$. The quantity $\rho(g,r,d)$ can therefore be regarded as the number of moduli for non-degenerate maps up to projective linear transformations; such is the data of a (base-point-free) linear series on $C$. We refer to \cite[IV-V]{acgh} for an introduction to Brill-Noether theory, and \cite{vogt} for a survey of recent developments. We review the basic notions we need in \S \ref{sec:BN}.

It is natural to consider existence theorems for maps from general curves to other target varieties. Among many other applications, the divisorial loci of curves possessing maps to $\bP^r$ with $\rho(g,r,d)=-1$ were used in \cite{eh_kodaira} to prove that $\cM_g$ is of general type for large $g$. Analogous results for maps to quadrics and Grassmannians were used to similar effect more recently in \cite{fjp1,fjp2}. 

In this paper, we study existence of maps from a general curve $C$ to the blow-up $\Bl_{\bP^s}\bP^r$ of $\bP^r$ along a linear space of dimension $s$. The geometry of the moduli space of maps $f\colon C \to \Bl_{\bP^s}(\bP^r)$ has been studied in different guises by various authors; see, e.g. \cite{cm,farkas,cotterill,chz,ungureanu,farkas2,cl2,lsakran,cl_complete,cl_interpolation}, as well as \cite[Chapter VIII, \S3-4]{acgh} for closely related enumerative problems. The variety $\Bl_{\bP^s}\bP^r$ an example of a toric variety; we have previously raised the question of whether the classical Brill-Noether Theorem can be extended to the setting of maps to toric varieties in \cite{CL-surfaces}, giving an answer in dimension 2.

A map $f\colon C \to \Bl_{\bP^s}\bP^r$ depends on two numerical invariants
\[
(d,k) := \left(\deg(f^*\cO_{\bP^r}(1)), \deg(f^*\cO(E))\right),
\]
where $E \subset \Bl_{\bP^s}\bP^r$ is the exceptional divisor. We will assume that all maps have image not contained in $E$, so that $d \geq k \geq 0$. The pair $(d,k)$ is referred to as the \emph{degree} of $f$. A map of degree $(d,k)$ is equivalently a map $f\colon C \to \bP^r$ for which the blow-up center $\bP^s \subset \bP^r$ is a $k$-secant plane.

\begin{definition}\label{def:degenerate_nice}
Let $C$ be a smooth curve. A map $f\colon C \to \Bl_{\bP^s}\bP^r$ is \emph{non-degenerate} if it is so after post-composing with both the blow-up morphism $b\colon \Bl_{\bP^s}\bP^r \to \bP^r$ and the projection $\pi\colon \Bl_{\bP^s}\bP^r \to \bP^{r-s-1}$.

We say that $f$ is \emph{nice} if, in addition, $H^1(C,f^*T_{\Bl_{\bP^s}\bP^r})=0$.
\end{definition}

\begin{remark}\label{rem:degenerate}
In fact, if \(b\circ f\) is non-degenerate, then \(\pi\circ f\) is automatically non-degenerate as well. However, we have included both requirements in the definition of non-degeneracy of \(f\) for emphasis.
\end{remark}

Our main theorem is a complete characterization for existence of nice maps.

\begin{theorem}\label{thm:main}
Let $C$ be a general curve of genus $g$. Let $r,s$ be non-negative integers with $s \leq r-2$. Then, a nice map $f\colon C \to \Bl_{\bP^s}\bP^r$ of degree $(d,k)$ exists if and only if the following three inequalities hold.
\begin{enumerate}
\item[(1)] $\rho(g,r,d) := g-(r+1)(g-d+r)\ge0$,
\item[(2)] $\rho_{\bullet}(g,d,k;r,s) := \rho(g,r,d)+(r-s)(s+1)-(r-s-1)k\ge0$, and
\item[(3)] $\rho'(g,d,k;r,s) := \rho(g,r-s-1,d-k)\ge0$.
\end{enumerate}
\end{theorem}

Conditions (1) and (3) are both necessary by the classical Brill-Noether existence theorem. Indeed, post-composing $f$ with the morphisms $b,\pi$ gives non-degenerate maps $C \to \bP^r$ and $C \to \bP^{r-s-1}$. The condition (2) is necessary because, if $H^1(C,f^*T_{\Bl_{\bP^s}\bP^r})=0$, then the moduli space of maps $f\colon C \to \Bl_{\bP^s}\bP^r$ of degree $(d,k)$ is smooth of dimension
\[
h^0(C,f^*T_{\Bl_{\bP^s}\bP^r}) = \rho_{\bullet}(g,d,k;r,s)+\dim(\mathrm{Aut}(\Bl_{\bP^s}\bP^r)).
\]
Thus, the sufficiency of conditions (1)-(3) is the interesting direction of Theorem \ref{thm:main}. 

\begin{remark}\label{rem:not_nice}
In fact, in characteristic zero, whenever the moduli space of non-degenerate maps $f\colon C \to \Bl_{\bP^s}\bP^r$ is non-empty, it is known to be pure of the expected dimension (\cite[Theorem 0.1]{farkas}; see also \cite[Corollary 2.3.4]{cl_complete} or Theorem \ref{thm:farkas} below). Thus, conditions (1)-(3) are equivalently necessary and sufficient with the requirement that $f$ be non-degenerate, but not necessarily nice. In positive characteristic, it could be true in principle that non-degenerate, but non-nice maps for which $\rho_{\bullet}<0$ exist, though this would be ruled out with a proof of Theorem \ref{thm:farkas} in arbitrary characteristic. We also do not know, even in characteristic 0, whether any non-degenerate map is automatically nice, which would be an analog of the Gieseker-Petri theorem.
\end{remark}

In the language of linear series with unexpected secants, a partial result toward Theorem \ref{thm:main} was proven by Farkas in \cite[Theorem 0.5]{farkas}. Namely, the necessary conditions (1)-(3) were given, but sufficiency was only proven under additional numerical assumptions. Our theorem provides an optimal improvement of Farkas's result.

\begin{remark}\label{rem:notation}
We have chosen to diverge in our notation from \cite{farkas}, as follows. We have indexed our invariants with the integers $g,d,k;r,s$, whereas \cite{farkas} uses instead $g,r,d$, in addition to $e=k$ and $f=k-s-1$. The quantity $f$ is natural for the following reason: a non-degenerate curve in $\bP^r$ automatically possesses an $(s+1)$-secant $s$-plane, so $f=k-s-1$ measures the extent to which an unexpected secant is present. However, our arguments will be uniform in $k$, allowing in particular $k<s+1$. (See below for a discussion of the case $k=0$, which recovers the classical Brill-Noether theorem.) Moreover, we take the perspective of studying the moduli space of maps to a fixed blow-up $\Bl_{\bP^s}\bP^r$, as opposed to maps to projective space with a fixed excess degree of secancy.
\end{remark}

We prove our existence theorem, Theorem \ref{thm:main}, by degeneration. Because niceness is an open condition in moduli, it suffices to produce, given $(g,d,k;r,s)$, a nice map out of \emph{one} curve. Whereas Farkas's partial result is achieved via degenerations of limit linear series, our approach uses instead degenerations of stable maps, which in particular allow us to pass through curves that are not of compact type. The technique is quite flexible, and we expect it to be applicable to many other target varieties.

\begin{example}\label{eg:intro_example}
Let $(g,d,k;r,s)=(8,10,5;3,1)$. Then, Theorem \ref{thm:main} asserts the existence of a map $f\colon C \to \Bl_{\bP^1}\bP^3$ of degree $(10,5)$ out of a general curve of genus $8$. Equivalently, $f$ is a space curve of degree $10$ with a $5$-secant line. The existence of such an $f$ is not immediately given by \cite[Theorem 0.5]{farkas}.

We prove the existence of such an $f$ as follows. We construct a map
\[
f\colon C' \cup R_1 \cup R_2 \to \Bl_{\bP^1}\bP^3,
\]
where $C' \to \Bl_{\bP^1}\bP^3$ is a nice map of degree $(6,3)$ out of a general curve of genus $4$, and the maps $R_i \to \Bl_{\bP^1}\bP^3$ are maps of degree $(2,1)$ out of rational curves $R_i \cong \bP^1$. The rational components are attached to $C'$ at $3$ nodes each. The map $C' \to \Bl_{\bP^1}\bP^3$ may be regarded as a canonically embedded
\[
C' \hookrightarrow \bP^1 \times \bP^1 \hookrightarrow \bP^3,
\]
where the blow-up center $\bP^1 \subset \bP^3$ is a ruling of the quadric surface $\bP^1 \times \bP^1 \hookrightarrow \bP^3$. An analysis of the restricted tangent bundle $T_{\Bl_{\bP^1}\bP^3}|_{C'}$ will show that $f$ may be deformed to a nice map of degree $(10,5)$ out of a smooth curve of genus $8$.
\end{example}

Degenerations such as that in Example \ref{eg:intro_example} allow us to reduce Theorem \ref{thm:main}, for any given $r,s$, either to the case $g=0$ or finitely many exceptional base cases. When $k=0$, our proof specializes to a proof of the classical Brill-Noether existence theorem, in which case $g=0$ suffices as a base case. In fact, the $k=0$ proof is implicit in \cite{larson}, which establishes the stronger property of \emph{interpolation} for the restricted tangent bundle of a general map to projective space.

The general case with secants contains several new features. For example, the construction of suitable degenerate curves to establish inductive steps is significantly more delicate, see \S\ref{sec:tree_desiderata}-\ref{sec:tree_construction}. The analysis of base cases is also more complex. When $k=0$, the only base case needed in the proof of classical Brill-Noether existence is the existence of rational normal curves in $\bP^r$, whereas existence with secants requires separate arguments to establish existence for four exceptional families of curves of positive genus.

Our proof of Theorem \ref{thm:main} is valid in arbitrary characteristic. However, it is simpler in characteristic $0$, where Theorem \ref{thm:farkas}, guaranteeing that the moduli space of non-degenerate maps has expected dimension, is available. The paper is structured as follows: after preliminaries in \S \ref{sec:preliminaries}, we set up our inductive arguments in \S \ref{sec:inductive_steps} via constructions of degenerate curves. We reduce Theorem \ref{thm:main} to finitely many base cases in \S \ref{sec:reductions}, and handle the base cases in \S \ref{sec:base_cases}. The simplification afforded by assuming $\charac(\bK)=0$ enters only in \S \ref{sec:terminal_base}. For clarity of exposition, we first complete the proof of Theorem \ref{thm:main} under this assumption, and then explain in \S \ref{sec:positive_char} how to circumvent this assumption with additional ad hoc arguments. While a more satisfying resolution would be to extend Theorem \ref{thm:farkas} to arbitrary characteristic, we do not currently have a proof.

\subsection{Acknowledgements}

The second-named author learned of the proof of the classical Brill-Noether existence theorem given in \cite{larson} at the SLMath Summer School on Algebraic Curves in 2024, organized by Izzet Coskun, Eric Larson, Hannah Larson, and Isabel Vogt. He is grateful to the organizers for inviting him to participate in this event. The second-named author also thanks the University of Cambridge for its hospitality during his visit in summer 2026, where work on this project took place. We are grateful to Gavril Farkas, Nathan Ilten, Dave Jensen,  Alberto Landi, Hannah Larson, Chi Ki Ngai, Dhruv Ranganathan, and Karolyn So for explaining their results to us and for helpful conversations. A.C. was supported SNF grant P500PT-222363  and EPSRC Horizon Europe Guarantee
EP/Y037162/1. C.L. is supported by NSF Grant DMS-2601731.

The authors used AI for help with \LaTeX{} typesetting, and to check numerical statements. Specifically, some numerical inequalities appearing in \S\ref{sec:reductions} were initially verified with AI assistance, and AI was used to determine whether various special cases of Theorem \ref{thm:main} satisfy the numerical hypotheses of \cite[Theorem 0.5]{farkas}. The observation that Serre duality immediately implies existence of an ILS in Proposition \ref{prop:serre_III} was obtained with AI assistance, in a special case. The authors subsequently realized that passing to Serre duals, in conjunction with the inductive arguments of \S\ref{sec:inductive_steps} and additional base-point-freeness arguments, suffices to obtain the stronger conclusion of existence of nice maps in most base cases. All statements and proofs, as they appear in the paper, were developed and written by the authors, who take full responsibility for their correctness.

\section{Preliminaries}
\label{sec:preliminaries}

In this section, we collect the needed results and notation regarding maps from general curves to $\bP^r$ and $\Bl_{\bP^s}(\bP^r)$. Throughout the paper, a \emph{curve} is always assumed to be proper, connected, at-worst-nodal, and defined over an algebraically closed field $\bK$. The \emph{genus} of a curve refers to its arithmetic genus. In this section, all curves are smooth.

\subsection{Brill-Noether theory}
\label{sec:BN}

Let $C$ be a smooth curve, and let $r \geq 1$ be an integer. Recall that a map $f\colon C \to \bP^r$ is \emph{non-degenerate} if its image is not contained in a hyperplane. Let $\cM_d(C,\bP^r)$ be the moduli space of maps $f\colon C \to \bP^r$ of degree $d$, and let $\cM_d(C,\bP^r)^\circ \subset \cM_d(C,\bP^r)$ be the open subset of non-degenerate maps. A non-degenerate map of degree $d$ necessarily has $d \geq r$.

A non-degenerate map $f\colon C \to \bP^r$ is determined by the data of a line bundle $\cL$ of degree $d$ on $C$, along with $r+1$ linearly independent sections $f_0,\ldots,f_r \in H^0(C,\cL)$, taken up to scaling. The linear span
\[
W = \langle f_0,\ldots,f_r\rangle \subset H^0(C,\cL)
\]
is a vector subspace of dimension $r+1$. The data $(\cL,W)$ of a line bundle of degree $d$ and a subspace $W \subset H^0(C,\cL)$ is called a \emph{linear series} of degree $d$ and rank $r$. The moduli space of linear series on $C$ is denoted $G^r_d(C)$, and is projective.

The \emph{Brill-Noether theorem} \cite[V.1.5]{acgh} asserts that, if $C$ is general, then $G^r_d(C)$ is pure of dimension
\[
\rho(g,r,d) := g - (r+1)(g-d+r).
\]
If $\rho(g,r,d)<0$, then $G^r_d(C)$ is empty. The \emph{Brill-Noether existence theorem} \cite[V.1.1]{acgh}, also Theorem \ref{thm:BN_existence} above, asserts furthermore that $G^r_d(C)$ is non-empty if $\rho(g,r,d) \geq 0$.

$G^r_d(C)$ may be cut out as a degeneracy locus inside a Grassmannian bundle over the Picard variety $\Pic^d(C)$, of expected dimension $\rho(g,r,d)$. In particular, the Brill-Noether number $\rho(g,r,d)$ is a priori a lower bound for $\dim(G^r_d(C))$, and the Brill-Noether theorem asserts that this expected dimension is the actual dimension when $C$ is general. The \emph{Gieseker-Petri theorem} \cite[V.1.6]{acgh} asserts furthermore that $G^r_d(C)$ is smooth.

Let $\psi\colon \cM_d(C,\bP^r)^\circ \to G^r_d(C)$ be the map remembering the underlying linear series of a non-degenerate map $f$. Then, the image of $\psi$ is the locus of \emph{base-point-free} (bpf) linear series $W$, which are those without a common point $p \in C$ of vanishing. The bpf locus in $G^r_d(C)$ is open and dense, because the locus of linear series with a base point at a \emph{given} point $p \in C$ has dimension
\[
\rho(g,r,d-1) \leq \rho(g,r,d)-2.
\]

The fiber of $\psi$ over a bpf linear series $W$ is identified with the set of bases of $W$, taken up to scaling. In particular, the morphism $\psi$ is smooth of relative dimension
\[
\dim(\Aut(\bP^r)) = (r+1)^2 - 1,
\]
so $\cM_d(C,\bP^r)^\circ$ is smooth of dimension
\[
\rho(g,r,d) + (r+1)^2 - 1 = (r+1)d - rg + r.
\]
The tangent space to a point $f \in \cM_d(C,\bP^r)^\circ$ is identified with $H^0(C,f^*T_{\bP^r})$.

Finally, let $c\colon G^r_d(C) \to \Pic^d(C)$ be the map remembering the underlying line bundle of a linear series. If $d \geq g+r$, then $c$ is surjective, by Riemann-Roch. If instead $d \leq g+r$ (the ``Brill-Noether special'' range), then the image of $c$ is denoted $W^r_d(C) \subset \Pic^d(C)$, and is the locus of line bundles $\cL$ with $h^0(C,\cL) \geq r+1$. The restriction $c\colon G^r_d(C) \to W^r_d(C)$ is surjective, birational, and an isomorphism over $W^r_d(C) \setminus W^{r+1}_d(C)$.

\subsection{Brill-Noether theory with secants}\label{sec:BN_secant}

Our interest in this paper is to extend the results of \S \ref{sec:BN} to the setting of maps to $\Bl_{\bP^s}(\bP^r)$. Fix non-negative integers $r,s$ with $s \leq r-2$, and a general curve $C$ of genus $g$. Write $X_{r,s} = \Bl_{\bP^s}(\bP^r)$. A non-degenerate (Definition \ref{def:degenerate_nice}) map $f\colon C \to X_{r,s}$ of degree $(d,k)$ gives rise to an ILS on $C$.

\begin{definition}\cite[Definition 4.2]{chz}
\label{def:ILS}
Let $C$ be a smooth curve of genus $g$. Fix integers $d \geq k \geq 0$, and $r,s$ with $-1 \leq s \leq r-1$. An \emph{inclusion of linear series} (\emph{ILS}) $(\cL,D,V \subset W)$ on $C$ consists of the following data.
\begin{itemize}
\item A linear series $(\cL,W) \in G^r_d(C)$,
\item an effective divisor $D \subset C$ of degree $k$, and
\item a subspace \(V \subset W\) of dimension \(r-s\) whose sections vanish along \(D\), giving a linear series $(\cL(-D),V) \in G^{r-s-1}_{d-k}(C)$.
\end{itemize}
The moduli space of ILS on $C$ is denoted $G^{r,r-s-1}_{d,k}(C)$.
\end{definition}

Indeed, when $0 \leq s \leq r-2$, from $f\colon C \to X_{r,s}$, the post-composition $b\circ f\colon C \to \bP^r$ gives rise to $(\cL,W) \in G^r_d(C)$, and the divisor $D \subset C$ is the pullback of the exceptional divisor, $D = f^{-1}(E)$. Then, the post-composition $\pi \circ f\colon C \to \bP^{r-s-1}$ has degree $d-k$, so gives rise to $(\cL(-D),V) \in G^{r-s-1}_{d-k}(C)$. Given a choice of non-zero section $1_D \in H^0(C,\cO(D))$ vanishing along $D$, we obtain an inclusion $V \subset W$. 

While we assume that $0 \leq s \leq r-2$ in order to be able to speak of the blow-up of $\bP^r$ at $\bP^s$, it is useful to allow $s=-1,r-1$ in the definition of an ILS. When $s=r-1$, the space $G^0_{d-k}(C)$ is simply the space of effective divisors of degree $d-k$ on $C$. When $s=-1$, the divisor $D \subset C$ must consist of base points of the linear series $W$.

Let $\cM_{(d,k)}(C,X_{r,s})^\circ$ be the moduli space of non-degenerate maps from $C$ to $X_{r,s}$ of degree $(d,k)$. Then, we have a map $\psi\colon \cM_{(d,k)}(C,X_{r,s})^\circ \to G^{r,r-s-1}_{d,k}(C)$, remembering the underlying ILS. Its image is the locus of ILS that are \emph{base-point-free}.

\begin{definition}
\label{def:bpf}
An ILS $(\cL,D,V \subset W) \in G^{r,r-s-1}_{d,k}(C)$ is \emph{base-point-free} (bpf) if both $(\cL,W) \in G^r_d(C)$ and $(\cL(-D),V) \in G^{r-s-1}_{d-k}(C)$ are. Equivalently, the following two conditions both hold.
\begin{enumerate}
\item[(BPF1)] $(\cL(-D),V) \in G^{r-s-1}_{d-k}(C)$ is bpf.
\item[(BPF2)] For all $p \in D$, not every section of $W \subset H^0(C,\cL)$ vanishes at $p$.
\end{enumerate}
\end{definition}

In the case of sections $v \in V$, we make a distinction between vanishing at $p \in C$ when viewed as an element of $H^0(C,\cL(-D))$ or $H^0(C,\cL)$. The requirement (BPF1) is that there is no point $p \in C$ for which every $v \in V$ vanishes when viewed as a section of $\cL(-D)$.

The map $\psi\colon \cM_{(d,k)}(C,X_{r,s})^\circ \to G^{r,r-s-1}_{d,k}(C)$ surjects onto the bpf locus. The non-empty fibers of $\psi$ are smooth of dimension
\[
\dim(\Aut(X_{r,s})) = (r-s)^2 + (s+1)(r+1) - 1,
\]
see \cite[Corollary 2.3.4]{cl_complete} for an explicit parameterization.

When $\charac(\bK)=0$, Farkas proved the following analogue of the Brill-Noether theorem; see \cite[\S 2.3]{cl_complete} for a translation into this particular formulation.

\begin{theorem}\label{thm:farkas}\cite[Theorem 0.1]{farkas}
Suppose that $\charac(\bK)=0$. Then, $G^{r,r-s-1}_{d,k}(C)$ is pure of dimension
\[
\rho_{\bullet}(g,d,k;r,s) := \rho(g,r,d) + (r-s)(s+1) - (r-s-1)k,
\]
the expected. If $\rho_{\bullet}(g,d,k;r,s)<0$, then $G^{r,r-s-1}_{d,k}(C)$ is empty.

In particular, $\cM_{(d,k)}(C,X_{r,s})^\circ$ is pure of dimension
\[
\rho_{\bullet}(g,d,k;r,s) + \dim(\Aut(X_{r,s})) = \chi(C,f^*T_{X_{r,s}}) = (r+1)d - (r-s-1)k - rg + r,
\]
the expected.
\end{theorem}

The moduli space of ILS $G^{r,r-s-1}_{d,k}(C)$ can be cut out as a degeneracy locus, whose dimension is bounded below by $\rho_{\bullet}(g,d,k;r,s)$, cf. \cite[\S 4]{chz}. Concretely, the value of the expected dimension, $\rho_{\bullet}(g,d,k;r,s)$, can be obtained as follows: given a linear series $(\cL,W) \in G^r_d(C)$, there is a
\[
\dim(\Gr(r-s,r+1)) = (r-s)(s+1)
\]
dimensional family of subspaces $V \subset W$. There is a $k$-parameter family of effective divisors $D \subset C$, and the condition that every section of $V$ vanish along $D$, as a section of $H^0(C,\cL)$, imposes $(r-s)k$ expected conditions. Because $\dim(G^r_d(C))=\rho(g,r,d)$, the formula for the expected dimension follows.

While the assumption on characteristic do not seem to be explicitly stated in \cite{farkas}, the proof of Theorem \ref{thm:farkas} uses it in an essential way. Indeed, after degeneration of \(C\) to a ``flag'' curve consisting of a \emph{tree} of rational curves attached to elliptic tails, the proof uses that the moduli space of linear series on \(\bP^{1}\) with ramification conditions at fixed points has expected dimension, which only holds in characteristic zero. While other proofs along analogous lines in Brill-Noether theory can be made to hold in arbitrary characteristic by replacing the degenerate curve with a chain of elliptic curves, we do not know how to do this in the setting of Theorem \ref{thm:farkas}. As such, we do not know whether Theorem \ref{thm:farkas} holds in arbitrary characteristic. We also do not know whether the natural analog of the Gieseker-Petri theorem, that $G^{r,r-s-1}_{d,k}(C)$ and $\cM_{(d,k)}(C,X_{r,s})^\circ$ are smooth, holds, in any characteristic.

Our main theorem, Theorem \ref{thm:main}, concerns the existence of components of $\cM_{(d,k)}(C,X_{r,s})^\circ$ which are generically smooth of expected dimension, in arbitrary characteristic. Indeed, a non-degenerate map $f$ is nice (Definition \ref{def:degenerate_nice}) if and only if it defines a point of $\cM_{(d,k)}(C,X_{r,s})^\circ$ that is smooth of expected dimension. Theorem \ref{thm:main} also does not rule out the possibility that $\cM_{(d,k)}(C,X_{r,s})^\circ$ contains everywhere non-reduced components of expected dimension.

In our proof of Theorem \ref{thm:main}, it will be useful to switch back and forth between the perspective of constructing smooth points of the moduli space of maps $\cM_{(d,k)}(C,X_{r,s})^\circ$ of expected dimension, and constructing smooth points of the moduli space of ILS $G^{r,r-s-1}_{d,k}(C)$ of expected dimension, which are in addition bpf. Our inductive arguments will primarily involve moduli spaces of maps, but the ILS perspective will be used to handle the base cases.

We close this section by setting some terminology that will be useful in the proof of Theorem \ref{thm:main}.

\begin{definition}
\label{def:valid}
Fix $r,s$ with $0 \leq s \leq r-2$. We say that a tuple $(g,d,k;r,s)$ is \emph{valid} if $g,k \geq 0$, and, in addition, the quantities
\[
\rho(g,r,d), \quad \rho_{\bullet}(g,d,k;r,s), \quad \rho'(g,d,k;r,s),
\]
appearing in Theorem \ref{thm:main}, are all non-negative. 

We will often abbreviate these quantities by $\rho,\rho_{\bullet},\rho'$, respectively. The triple $(\rho,\rho_{\bullet},\rho')$ is called the \emph{profile} of $(g,d,k;r,s)$. 

We write $\Ex(g,d,k;r,s)$ for the statement that Theorem \ref{thm:main} holds for $(g,d,k;r,s)$. When $r,s$ are fixed, we will often drop them from the notation, e.g., writing $\Ex(g,d,k)$.
\end{definition}

Thus, Theorem \ref{thm:main} is the statement that, if $(g,d,k;r,s)$ is valid, then $\Ex(g,d,k;r,s)$ holds. 

\begin{remark}\label{rmk:degree-checks}
    We record the following basic observation:
\[
\begin{aligned}
\rho \geq 0 &\implies d \geq r,\\
\rho' \geq 0 &\implies d-k \geq r-s-1.
\end{aligned}
\]
    In particular, if $(g,d,k;r,s)$ is valid, then the inequalities $d \geq r$ and $d-k \geq r-s-1$, which are necessary for non-degenerate maps to exist, are automatic.
\end{remark}

\section{Inductive steps}
\label{sec:inductive_steps}

\subsection{Setup}
\label{sc:inductive_setup}

Throughout this section, we fix \(r,s\), and write \(X=X_{r,s}\). Recall that we write \(\Ex(g,d,k)\) for the statement that \((g,d,k)\) is valid, and that there exists a nice map \(f\colon C \to X\) of degree \((d,k)\), where \(C\) is a general curve of genus \(g\). The main goal of this section is to prove the following, which will give a mechanism for establishing existence by induction.

\begin{proposition}
\,
\label{prop:inductive_steps_main}
\begin{enumerate}
    \item[(L)] If \(\Ex(g-1,d-1,k)\) holds, then \(\Ex(g,d,k)\) holds.
    \item[(R)] If \(\Ex(g-(r-s),d-(r-s),k-1)\) holds, then \(\Ex(g,d,k)\) holds.
    \item[\((T_\ell)\)] Let \(\ell\) be an integer with \(0\leq \ell\leq s+1\). If
    \[
    \Ex\bigl(g-(r+1)-\ell(r-s-1),d-r-\ell(r-s-1),k-\ell\bigr)
    \]
    holds, then \(\Ex(g,d,k)\) holds.
\end{enumerate}
\end{proposition}

The method of proof for all three statements will be to start with a nice map \(f'\colon C' \to X\), and construct a map \(f\colon C=C'\cup R \to X\), where \(R\) is a (possibly reducible) rational curve attached to \(C'\) at nodes, and \(f|_{C'}=f'\). Then, we will deform \(f\) to a nice map out of a smooth curve. The names of the statements in Proposition \ref{prop:inductive_steps_main} stand for \textbf{L}ine, \textbf{R}ational normal curve, and \textbf{T}ree, respectively, which describe the rational curves \(R\) used to prove the respective statements.

\begin{remark}
\label{rem:dimension_effects}
In each of the three cases of Proposition~\ref{prop:inductive_steps_main}, passing from \((g,d,k)\) to the predecessor (e.g. \((g-1,d-1,k)\)) has the following effects on the profile \((\rho,\rho_\bullet,\rho')\).
\begin{itemize}
    \item[(L)] \((\rho,\rho_\bullet,\rho')\to(\rho-1,\rho_\bullet-1,\rho'-1)\).
    \item[(R)] \((\rho,\rho_\bullet,\rho')\to(\rho-(r-s),\rho_\bullet-1,\rho')\).
    \item[\((T_\ell)\)] \((\rho,\rho_\bullet,\rho')\to(\rho-(r-s-1)\ell,\rho_\bullet,\rho'-(s+1-\ell))\).
\end{itemize}
Note that the values of \(\rho,\rho_\bullet,\rho'\) can only decrease under each of these decrements, as must be the case for the statements in Proposition~\ref{prop:inductive_steps_main} to be plausible.
\end{remark}

\begin{lemma}
\label{lem:attach_and_smooth}
Let \(f'\colon C'\to X\) be a nice map, and let \(R\) be a connected, nodal curve of arithmetic genus \(0\). Let \(C=C'\cup R\) be a nodal curve given by attaching \(C'\) and \(R\) at nodes \(C'\cap R=\{x_1,\ldots,x_m\}\). Let \(f\colon C\to X\) be any map for which \(f|_{C'}=f'\). Suppose that
\[
H^1\bigl(R,f^*T_X(-x_1-\cdots-x_m)\bigr)=0.
\]
Then, \(f\) deforms to a nice map.

In particular, write \(g'=g(C')\), and \((d',k')\) for the degree of \(f'\). Let \((\delta,\kappa)\) be the degree of \(f|_R\). Then, \(\Ex(g'+m-1,d'+\delta,k'+\kappa)\) holds.
\end{lemma}

\begin{proof}
We claim first that \(H^1(C,f^*T_X)=0\). Indeed, write \(\nu\colon C'\sqcup R\to C\) for the normalization map. Then, the short exact sequence
\[
0 \to f^*T_X \to \nu_*\nu^*(f^*T_X)
   \to \bigoplus_{i=1}^m (T_X)_{f(x_i)}
   \to 0.
\]
induces an exact sequence
\begin{equation}
\label{eq:exact_sequence_normalization}
H^0(C',f^*T_X)\oplus H^0(R,f^*T_X)\to \bigoplus_{i=1}^m (T_X)_{f(x_i)}\to H^1(C,f^*T_X)\to H^1(C',f^*T_X)\oplus H^1(R,f^*T_X)=0.
\end{equation}
The vanishing of the last term follows from the assumptions that \(f'\) is nice and that \(H^1\bigl(R,f^*T_X(-x_1-\cdots-x_m)\bigr)=0\). On the other hand, we also have an exact sequence
\[
H^0\bigl(R,f^*T_X(-x_1-\cdots-x_m)\bigr)\to H^0(R,f^*T_X)\to \bigoplus_{i=1}^m (T_X)_{f(x_i)}\to H^1\bigl(R,f^*T_X(-x_1-\cdots-x_m)\bigr)=0,
\]
so the first map in \eqref{eq:exact_sequence_normalization} is surjective. It follows that \(H^1(C,f^*T_X)=0\).

Therefore, \(f\) defines a smooth point of expected dimension for the morphism
\[
\barM_{g+m-1,0}(X,(d'+\delta,k'+\kappa))\to \mathfrak{M}_{g+m-1},
\]
where $\barM_{g+m-1,0}(X,(d'+\delta,k'+\kappa))$ is the moduli space of stable maps and $\mathfrak{M}_{g+m-1}$ is the moduli space of prestable curves. We conclude that there exists an open neighborhood of \(f\) dominating \(\mathfrak M_{g+m-1}\), and containing nice maps of degree \((d'+\delta,k'+\kappa)\). Note in particular that a general deformation of $f$ remains non-degenerate.
\end{proof}

\begin{remark}
\label{rem:make_nodal_curve}
In the setting of Lemma~\ref{lem:attach_and_smooth}, suppose the data of \(f'\colon C'\to X\) is given, along with points \(x_1,\ldots,x_m\in C'\). Then, constructing a map \(f\colon C=C'\cup R\to X\) amounts to finding a nodal rational curve \(R\) with smooth points (abusively denoted) \(x_1,\ldots,x_m\in R\), along with a map \(h\colon R\to X\) with \(f'(x_i)=h(x_i)\) for all \(i=1,\ldots,m\). Indeed, a map out of a nodal curve is equivalent to the data of a map out of its normalization, mapping preimages of any given node to the same point. Concretely, one needs to find a rational curve interpolating through (the images under \(f'\) of) the points \(x_1,\ldots,x_m\) in \(X\).
\end{remark}

Fix a nice, and in particular non-degenerate, map \(f'\colon C'\to X\). Throughout this section, we will, by abuse of notation, freely identify points \(x\in C'\) with their images \(f'(x)\in X\), and in turn with their images under $b:X\to \bP^r$. While \(f'\) is not assumed to be injective, these identifications will be harmless. We will repeatedly use the following lemma.

Recall that \(x_1,\ldots,x_m\in \bP^r\) are said to be in \emph{linearly general position} if any collection of \(r'\leq r+1\) of the points \(x_j\) spans a linear space of dimension exactly \(r'-1\). More generally, given a linear subspace $\Pi \subseteq \bP^r$ we say that  \(x_1,\ldots,x_m\in \bP^r\) are in linearly general position with respect to $\Pi$ if any any collection of \(r'\leq r+1\) of the points \(x_j\) spans a linear space intersecting $\Pi$ in the expected dimension. Equivalently, their images under the projection map from $\Pi$ are in linearly general position. 

\begin{lemma}
\label{lem:curve_linearly_general}
Let \(x_1,\ldots,x_m\in C'\) be general points and $\Pi_1,\ldots, \Pi_\ell \subseteq \bP^r$ be any collection of linear spaces. Then, (the images under \(f'\) of) the points \(x_1,\ldots,x_m\) in \(X\) are in linearly general position with respect to each $\Pi_1,\ldots, \Pi_\ell$.

In particular, if \(m\leq r-s\), then the linear span of the \(x_j\) in \(\bP^r\) does not meet the blown-up linear space \(\bP^s\subset \bP^r\), and if \(m=r-s+1\), then \(\langle x_1,\ldots,x_{r-s+1}\rangle\cap \bP^s\) is a single point.
\end{lemma}

\begin{proof}
Let \(s_i= \mathrm{dim}(\Pi_i)\), and let \(\pi_i: \bP^r \dashrightarrow \bP^{r-s_i-1}\) be the projection from \(\Pi_i\), for \(i=1,\ldots, \ell\). The condition of linear generality is manifestly Zariski open, so it suffices to construct one collection of points \(x_1,\ldots,x_m\in C'\) with the desired property. This can be done inductively as follows. At the $j$-th step, choose \(x_j\in C'\) in such a way that it does not lie in the linear span of any \(r\) of the previously chosen points in \(\bP^r\), or any \(r-s_i-1\) of the previously chosen points in \(\bP^{r-s_i-1}\), for \( i=1,\ldots, \ell\). This is possible exactly because the image of \(f'\) is not contained in any hyperplane in \(\bP^r\) or \(\bP^{r-s_i-1}\).
\end{proof}

\subsection{(L) and (R)} \label{sec:inductive_easy}

\begin{proposition}[Proposition \ref{prop:inductive_steps_main}(L)] \label{prop:L}
 If \(\Ex(g-1,d-1,k)\) holds, then \(\Ex(g,d,k)\) holds.
\end{proposition}

\begin{proof}
Let \(f'\colon C'\to X\) be a nice map of degree \((d-1,k)\), where the genus of \(C'\) is \(g-1\). Let \(x_1,x_2\in C'\) be general points. By Lemma \ref{lem:curve_linearly_general}, the unique line through \(x_1,x_2\) in \(\bP^r\) does not meet the blow-up center \(\bP^s\subset \bP^r\). Therefore, we may take \(h\colon R=\bP^1\to X\) to be the inclusion of this line passing through \(x_1,x_2\) in \(X\), of degree \((1,0)\).

The restricted tangent bundle \(h^*T_X\) is the same as the restricted tangent bundle of the inclusion of a line \(h\colon \bP^1\to \bP^r\), which is
\[
h^*T_{\bP^r}\simeq \cO(1)^{\oplus r-1}\oplus \cO(2),
\]
see e.g. \cite[Proposition 3.1]{larson} or Lemma \ref{lem:tangent_degenerate_curve} below. Therefore, taking \(C=C'\cup \bP^1\) to be the curve obtained by gluing \(C'\) and \(\bP^1\) at nodes \(x_1,x_2\), the maps \(f'\), \(h\) together induce a map \(f\colon C\to X\), satisfying the hypothesis of Lemma \ref{lem:attach_and_smooth}. The conclusion follows.
\end{proof}

\begin{proposition}[Proposition \ref{prop:inductive_steps_main}(R)] \label{prop:R}
If \(\Ex(g-(r-s),d-(r-s),k-1)\) holds, then \(\Ex(g,d,k)\) holds.
\end{proposition}

\begin{proof}
Let \(f'\colon C'\to X\) be a nice map of degree \((d-(r-s),k-1)\), where the genus of \(C'\) is \(g-(r-s)\). Let \(x_1,\dots,x_{r-s+1}\in C'\) be general points. By Lemma \ref{lem:curve_linearly_general}, the linear span \(\Lambda\cong \bP^{r-s}\subset \bP^r\) of \(x_1,\dots,x_{r-s+1}\) meets the blow-up center \(\bP^s\subset \bP^r\) in a unique point \(q\). Moreover, the points \(x_1,\dots,x_{r-s+1},q\in \Lambda\) are themselves in linearly general position, or else the linear span of a strict subset of \(x_1,\dots,x_{r-s+1}\) would meet \(\bP^s\) non-trivially, contradicting Lemma \ref{lem:curve_linearly_general}.

Therefore, there exists an embedded rational normal curve \(R\subset \Lambda\), passing through \(x_1,\dots,x_{r-s+1},q\). Moreover, \(R\) meets \(\bP^s\subset \bP^r\) only in the point \(q\), because \(\Lambda\cap \bP^s=\{q\}\). It follows that \(R\subset \Lambda\) lifts to a map \(h\colon R=\bP^1\to X\) of degree \((r-s,1)\), whose image contains \(x_1,\dots,x_{r-s+1}\). Furthermore, we have
\[
h^*T_X\cong \cO(r-s)^{\oplus r-1}\oplus \cO(r-s+1),
\]
by Lemma \ref{lem:tangent_r-s} below. We now conclude by Lemma \ref{lem:attach_and_smooth}.
\end{proof}

\begin{lemma} \label{lem:tangent_r-s}
Let \(h\colon \bP^1\to X\) be a map of degree \((r-s,1)\), given by the proper transform of a rational normal curve \(R\subset \Lambda\cong \bP^{r-s}\subset \bP^r\), where \(\Lambda\cap \bP^s=\{q\}\). Then,
\[
h^*T_X\cong \cO(r-s)^{\oplus r-1}\oplus \cO(r-s+1).
\]
Furthermore, let \(x\in \bP^1\) be any point not mapping to the exceptional divisor of $X$. Then, the fiber of the summand \(\cO(r-s+1)\subset h^*T_X\) at $x$ corresponds to the tangent direction at \(x\in \bP^r\) pointing toward \(q\).
\end{lemma}

\begin{proof}
Let \(\widetilde{\Lambda}\cong \Bl_q(\bP^{r-s})\subset X\) be the proper transform of \(\Lambda\subset \bP^r\). Then, we have a short exact sequence
\[
0\to T_{\widetilde{\Lambda}}|_R\to h^*T_X\to N_{\widetilde{\Lambda}/X}|_R\to 0.
\]
Because \(\widetilde{\Lambda}\subset X\) is cut out by the proper transforms of \(s\) general hyperplanes in \(\bP^r\), we have
\[
N_{\widetilde{\Lambda}/X}|_R\cong \cO_{\bP^r}(1)^{\oplus s}|_R\cong \cO(r-s)^{\oplus s}.
\]

Next, let \(\pi\colon \widetilde{\Lambda}\to \bP^{r-s-1}\) be the projection. Then, we have a short exact sequence
\[
0\to T_{\widetilde{\Lambda}/\bP^{r-s-1}}|_R\to T_{\widetilde{\Lambda}}|_R\to T_{\bP^{r-s-1}}|_R\to 0.
\]
The image of \(R\) in \(\bP^{r-s-1}\) is a rational normal curve, and any two rational normal curves are projectively equivalent. We conclude (e.g. by \cite[Proposition 3.2]{larson}) that
\[
T_{\bP^{r-s-1}}|_R\cong \cO(r-s)^{\oplus r-s-1}.
\]
By comparing with the degree of \(T_{\widetilde{\Lambda}}|_R\), we deduce that \(T_{\widetilde{\Lambda}/\bP^{r-s-1}}|_R\cong \cO(r-s+1)\). Therefore, we have
\[
\operatorname{Ext}^1(T_{\bP^{r-s-1}}|_R,T_{\widetilde{\Lambda}/\bP^{r-s-1}}|_R)=0
\Rightarrow
T_{\widetilde{\Lambda}}|_R\cong \cO(r-s)^{\oplus r-s-1}\oplus \cO(r-s+1).
\]
We have in turn that
\[
\operatorname{Ext}^1(N_{\widetilde{\Lambda}/X}|_R,T_{\widetilde{\Lambda}}|_R)=0
\Rightarrow
h^*T_X\cong \cO(r-s)^{\oplus r-1}\oplus \cO(r-s+1),
\]
as needed.

Finally, the summand \(\cO(r-s+1)\) is precisely the summand coming from the relative tangent bundle \(T_{\widetilde{\Lambda}/\bP^{r-s-1}}|_R\), which is naturally identified with the bundle of tangent directions pointing toward \(q\).
\end{proof}

\subsection{(\(T_\ell\)): Desiderata}
\label{sec:tree_desiderata}

In the remainder of this section, we prove Proposition \ref{prop:inductive_steps_main}$(T_\ell)$. As in the cases \((L)\) and \((R)\), the strategy is to choose \(m=r+2+\ell(r-s-1)\) general points \(x_1,\dots,x_m\in C'\), and find a rational curve \(h\colon R\to X\) of degree \((r+\ell(r-s-1),\ell)\), passing through the \(x_j\), for which
\[
H^1(R,h^*T_X(-x_1-\cdots-x_m))=0.
\]
If the images of the \(x_j\) in \(X\) were general points, then \cite[Theorem 1.5.1]{cl_interpolation} would prove the existence of such a rational curve.\footnote{In fact, the ``Tevelev degree'' enumerating such maps \(h\), which are in addition required to pass through the \(x_i\in X\) at \emph{prescribed} points \(p_i\in R\), is equal to \(\binom{s+1}{\ell}\), see \cite[Theorem 1.5.2]{cl_interpolation} and \cite[\S 4.3]{lsakran}.}

However, the points \(x_j\in C'\) being required to lie in the image of a curve, existence results for rational curves passing through general points of \(X\) do not immediately apply. We will circumvent the issue by constructing instead a \emph{reducible} curve \(R\) with the needed properties. The advantage will be that, if the degree of \(h\) upon restriction to each component of \(R\) is small, then we will be able to control the restricted tangent bundle \(h^*T_X\). The precise shape of the reducible rational curve we will construct is given below.

The edge case \(\ell=0\) is exceptional, and easier, so we will assume for now that \(\ell>0\) and defer the remaining case to the end.

\begin{situation}
\label{sit:curve_many_points}
Let \(C'\) be a smooth curve, and let \(f'\colon C' \to X\) be a non-degenerate map, fixed throughout.

Let \(\ell\) be an integer with \(1 \leq \ell \leq s+1\). Let \(R=R_0\cup\cdots\cup R_\ell\) be a nodal chain of rational curves, with \(R_i\cong \bP^1\) and \(y_i:=R_i\cap R_{i+1}\). Fix distinct points:
\begin{itemize}
\item \(p_{ij}\in R_i\), for \(1\leq i\leq \ell-1\) and \(1\leq j\leq r-s\), and
\item \(p_{\ell j}\in R_\ell\), for \(1\leq j\leq r-s+1\).
\end{itemize}
\end{situation}

\begin{proposition}
\label{prop:tree_exists}
In Situation~\ref{sit:curve_many_points}, there exist points \(p_{01},\ldots,p_{0,r-\ell+1}\in R_0\), points \(x_{ij}\in C'\) (identified with their images in \(X\)), and a map \(h\colon R\to X\), with the following properties.
\begin{enumerate}
\item \(h(p_{ij})=x_{ij}\) for all \(i,j\),
\item \(\deg(h^0)=(r-\ell,0)\) and \(\deg(h^i)=(r-s,1)\) for all \(i\geq 1\), where \(h^i\) is the restriction of \(h\) to \(R_i\), and
\item \(H^1(R,h^*T_X(-\sum p_{ij}))=0\).
\end{enumerate}
\end{proposition}

Note that it would have sufficed to show that \emph{there exist} points \(p_{ij}\in R_i\), for \(1\leq i\leq \ell\) as in Situation~\ref{sit:curve_many_points}, but the key idea of our construction is to fix these points initially in order to gain better control on the eventual map \(h\). The desired condition \(H^1(R,h^*T_X(-\sum p_{ij}))=0\) implies, in any case, that one should expect \(h\) to exist given any general choices of fixed points \(p_{ij}\in R\) and \(x_{ij}\in C'\).

The strategy of the construction is as follows. Begin by choosing points \(x_{\ell 1},\ldots,x_{\ell,r-s+1}\in X\) in linearly general position. Then, there is a 1-dimensional family \(\cF_\ell\) of maps \(h^\ell\colon R_\ell\to X\) of degree \((r-s,1)\) with the property that \(h^\ell(p_{\ell j})=x_{\ell j}\) for \(j=1,\ldots,r-s+1\). Varying over \(\cF_\ell\), we will see that the point \(h^\ell(y_{\ell-1})\) sweeps out a subvariety \(\Pi_{\ell-1}\setminus B_{\ell-1}\subset X\), where \(\Pi_{\ell-1}\subset X\) is a linear space of dimension \(1\) and \(B_{\ell-1}\) is a finite union of points. Moreover, the point \(h^\ell(y_{\ell-1})\in \Pi_{\ell-1}\setminus B_{\ell-1}\) uniquely determines the map \(h^\ell\in \cF_\ell\).

Next, choose points \(x_{\ell-1,1},\ldots,x_{\ell-1,r-s}\in X\) in linearly general position with respect to \(\Pi_{\ell-1}\) and (every point in) \(B_{\ell-1}\). There is a 2-dimensional family \(\cF_{\ell-1}\) of maps \(h^{\ell-1}\colon R_{\ell-1}\to X\) of degree \((r-s,1)\) with the property that:
\begin{itemize}
\item \(h^{\ell-1}(p_{\ell-1,j})=x_{\ell-1,j}\) for \(j=1,\ldots,r-s\), and
\item \(h^{\ell-1}(y_{\ell-1})\in \Pi_{\ell-1}\setminus B_{\ell-1}\).
\end{itemize}
Then, varying over \(\cF_{\ell-1}\), the point \(h^{\ell-1}(y_{\ell-2})\) sweeps out a subvariety \(\Pi_{\ell-2}\setminus B_{\ell-2}\subset X\), where \(\Pi_{\ell-2}\subset X\) is a linear space with dimension \(2\), and \(B_{\ell-2}\) is a finite union of linear spaces of strictly smaller dimension. Once more, the point \(h^{\ell-1}(y_{\ell-2})\in \Pi_{\ell-2}\setminus B_{\ell-2}\) uniquely determines the map \(h^{\ell-1}\in \cF_{\ell-1}\).

We will continue inductively in this fashion, until reaching a family \(\cF_1\) of maps \(h^1\colon R_1\to X\), where \(h^1(y_0)\) sweeps out the complement \(\Pi_0\setminus B_0\) of a finite union \(B_0\) of proper linear subspaces of a linear subspace \(\Pi_0\subset X\) of dimension \(\ell\). Finally, choose \(x_{01},\ldots,x_{0,r-\ell+1}\in X\) in linearly general position with respect to \(\Pi_0\) and all linear subspaces comprising \(B_0\). In particular, there is a unique point \(x\in \Pi_0\setminus B_0\) in the span \(\Lambda_0\) of the points \(x_{0j}\). For general points \(p_{01},\ldots,p_{0,r-\ell+1}\in R_0\), there exists a unique map \(h^0\colon R_0\to X\) of degree \((r-\ell,0)\) with the property that \(h^0(p_{0j})=x_{0j}\) for \(j=1,\ldots,r-\ell+1\), and \(h^0(y_0)=x\). 

Now, there exists a unique map \(h^1\in \cF_1\) with \(h^1(y_0)=h^0(y_0)=x\in \Pi_0\setminus B_0\). From here, there is a unique map \(h^2\in \cF_2\) for which \(h^2(y_1)=h^1(y_1)\in \Pi_1\setminus B_1\). Continuing in this fashion, we uniquely determine a sequence of maps \(h^0,\ldots,h^\ell\), which we will show glue to a global map \(h\colon R\to X\) satisfying all of the needed properties.

\subsection{(\(T_\ell\)): Construction}
\label{sec:tree_construction}

The precise statement that allows us to carry out the construction described in the previous section is the following.

\begin{lemma}
\label{lem:family_of_maps}
Let \(\Pi\subset X\) be a linear subspace\footnote{By a linear subspace, we mean the proper transform of a linear subspace of \(\bP^r\).} of dimension \(m\leq s\), with the property that the image of \(\Pi\) under the projection \(\pi\colon X\to \bP^{r-s-1}\) is a point. Let \(B\subset \Pi\) be a finite union of proper linear subspaces. Let \(x_1,\ldots,x_{r-s}\in X\) be points in linearly general position with respect to \(\Pi\) and all linear subspaces comprising \(B\). Fix a rational curve \(\bP^1\), and let \(y',p_1,\ldots,p_{r-s},y\in \bP^1\) be general points.

Let \(\cF\) be the family of maps \(h\colon \bP^1\to X\) of degree \((r-s,1)\) satisfying the following properties:
\begin{itemize}
\item \(h(p_j)=x_j\) for \(j=1,\ldots,r-s\), and
\item \(h(y)\in \Pi\setminus B\).
\end{itemize}
Then, we have the following.
\begin{enumerate}
\item[(a)] \(\dim(\cF)=m+1\).
\item[(b)] Varying over all maps \(h\in \cF\), the points \(h(y')\in X\) sweep out a subvariety \(\Pi'\setminus B'\subset X\), where \(\Pi'\subset X\) is a linear subspace of dimension \(m+1\), whose image under \(\pi\) is a point, and \(B'\) is a finite union of proper linear subspaces.
\item[(c)] The map \(\cF\to \Pi'\setminus B'\) given by evaluation at \(y'\) is an isomorphism of schemes.
\end{enumerate}
\end{lemma}

\begin{proof}
Let \(\overline{\cF}\) be the family of maps defined in the same way as \(\cF\), except that we relax the condition \(h(y)\in \Pi\setminus B\) to \(h(y)\in \Pi\). Clearly, \(\cF\subset \overline{\cF}\) is an open subset. We will explicitly parametrize \(\overline{\cF}\) in terms of local coordinates.

First, recall that a map \(h\colon \bP^1\to X\) of degree \((r-s,1)\) is given by a tuple
\[
h=[u_0u_1:\cdots:u_0u_{r-s}:u_{r-s+1}:\cdots:u_{r+1}],
\]
where
\begin{itemize}
\item \(u_0\in H^0(\bP^1,\cO(1))\), \(u_1,\ldots,u_{r-s}\in H^0(\bP^1,\cO(r-s-1))\), and \(u_{r-s+1},\ldots,u_{r+1}\in H^0(\bP^1,\cO(r-s))\),
\item the collections of sections \(u_1,\ldots,u_{r-s}\) and \(u_0,u_{r-s+1},\ldots,u_{r+1}\) both do not simultaneously vanish, and are both taken up to the same simultaneous scaling, and
\item (base-point-freeness) the sections \(u_0u_1,\ldots,u_{r+1}\in H^0(\bP^1,\cO(r-s))\) share no common factors.
\end{itemize}
For \(j=1,\ldots,r+1\), write \(e_j\in \bP^r\) for the point
\[
e_j=[0:\cdots:0:1:0:\cdots:0],
\]
where \(1\) appears in the \(j\)-th coordinate. For \(j<r-s\), the point \(e_j\in \bP^r\) may be identified with a point \(e_j\in X\). For \(j>r-s\), the point \(e_j\in \bP^r\) lies in the blown-up subspace \(\bP^s\subset \bP^r\). 

We may assume that the coordinates are chosen in such a way that
\begin{itemize}
\item \(x_j=e_j\), for \(j=1,\ldots,r-s\),
\item \(\Pi\subset X\) is the proper transform of \(\langle [1:\cdots:1],e_{r-s+1},\ldots,e_{r-s+m}\rangle\subset \bP^r\).
\end{itemize}
Choose homogeneous coordinates \(z,w\) on \(\bP^1\). Identify \(p_1,\ldots,p_{r-s}\in \bP^1\) with elements of \(\bK\subset \bP^1\), taking \(\frac{z}{w}\) to be an affine coordinate, and assume also that \(y'=\infty\) and \(y=0\). From the fact that \(h(p_j)=x_j\) for \(j=1,\ldots,r-s\), and that \(\pi(h(y))=\pi(h(0))=[1:\cdots:1]\in \bP^{r-s-1}\), we deduce that
\begin{align*}
u_1 &= p_1(z-p_2w)\cdots(z-p_{r-s}w),\\
u_2 &= p_2(z-p_1w)(z-p_3w)\cdots(z-p_{r-s}w),\\
&\vdots\\
u_{r-s} &= p_{r-s}(z-p_1w)(z-p_2w)\cdots(z-p_{r-s-1}w).
\end{align*}
 In particular, the sections \(u_1,\ldots,u_{r-s}\in H^0(\bP^1,\cO(r-s-1))\) are determined, as is the point
\[
\pi(h(y'))=\pi(h(\infty))=[p_1:\cdots:p_{r-s}]\in \bP^{r-s-1}.
\]

Next,  \(u_{r-s+1},\ldots,u_{r+1}\) must all vanish at \(p_1,\ldots,p_{r-s}\), and \(u_{r-s+m+1},\ldots,u_{r+1}\) must all be equal at \(y=0\). We conclude that we must have
\begin{align*}
u_{r-s+1} &= \gamma_1(z-p_1w)\cdots(z-p_{r-s}w),\\
&\vdots\\
u_{r-s+m} &= \gamma_m(z-p_1w)\cdots(z-p_{r-s}w),\\
u_j &= \gamma(z-p_1w)\cdots(z-p_{r-s}w), \qquad j=r-s+m+1,\ldots,r+1,
\end{align*}
for some \(\gamma_1,\ldots,\gamma_m,\gamma\in \bK\).

Finally, write \(u_0=\alpha z-\beta w\), for some \(\alpha,\beta\in \bK\). Then, because 
\[
u_0u_1,\ldots,u_0u_{r-s},u_{r-s+m+1},\ldots,u_{r+1}
\]
must all be equal at \(y=0\), we have
\[
(-\beta)(-1)^{r-s-1}p_1p_2\cdots p_{r-s}
=
\gamma(-1)^{r-s}p_1p_2\cdots p_{r-s}
\Longrightarrow \gamma=\beta.
\]

The base-point-free condition on the sections \(u_j\) amounts to the requirement that \(\alpha,\beta\) are not both \(0\), and furthermore that \(\frac{\alpha}{\beta}\neq p_1,\ldots,p_{r-s}\). Therefore, \(\overline{\cF}\) is identified (scheme-theoretically) with the open subset of the projective space \(\bP^{m+1}\) with homogeneous coordinates
\[
\alpha,\beta,\gamma_1,\ldots,\gamma_m,
\]
where \((\alpha,\beta)\neq(0,0)\) and \(\frac{\alpha}{\beta}\neq p_1,\ldots,p_{r-s}\).

Now, evaluation at \(y'=\infty\) gives a map
\begin{align*}
\ev_\infty\colon \overline{\cF} &\to X,\\
[\alpha:\beta:\gamma_1:\cdots:\gamma_m] &\mapsto [\alpha\cdot p_1:\cdots:\alpha\cdot p_{r-s}:\gamma_1:\cdots:\gamma_m:\beta:\cdots:\beta].
\end{align*}
It is clear by inspection that \(\ev_\infty\) is an isomorphism onto its image, which is the complement of a finite union of linear subspaces in a linear space \(\Pi'\subset X\) of dimension \(m+1\). Furthermore, we have \(\pi(\Pi')=[p_1:\cdots:p_{r-s}]\).

It remains to restrict \(\ev_\infty\) to the open subset \(\cF\subset \overline{\cF}\). However, \(\cF\subset \overline{\cF}\) is also the complement of a finite union of linear spaces, under the identification of \(\overline{\cF}\) with an open subset of \(\bP^{m+1}\). Indeed, for any \(u\in \overline{\cF}\), we have
\[
h(0)=[\beta\cdot 1:\cdots:\beta\cdot 1:\gamma_1:\cdots:\gamma_m:\beta:\cdots:\beta].
\]
In particular, removing \(h\in \overline{\cF}\) with the property that \(h(0)\in B\subset \Pi\) amounts to removing the loci on \(\overline{\cF}\) where certain linear equations in \(\beta,\gamma_1,\ldots,\gamma_m\) vanish. All three of the needed statements (a)-(c) follow.
\end{proof}

\begin{proof}[Proof of Proposition \ref{prop:tree_exists}]
We now carry out the construction of the desired map \(h\colon R\to X\), outlined at the beginning of this section. Choose points \(x_{\ell 1},\ldots,x_{\ell,r-s+1}\in C'\) in linearly general position. Apply Lemma \ref{lem:family_of_maps} with:
\begin{itemize}
\item \(\Pi=x_{\ell,r-s+1}\) and \(B=\emptyset\),
\item \(x_j=x_{\ell j}\) for \(j=1,\ldots,r-s\),
\item \(\bP^1=R_\ell\), and
\item \((y',y)=(y_{\ell-1},y_\ell)\).
\end{itemize}
Let \(\cF_\ell\) be the resulting family of maps \(h^\ell\colon R_\ell\to X\), and define \(\Pi_{\ell-1}=\Pi'\), \(B_{\ell-1}=B'\). Proceeding inductively in the order \(R_{\ell-1},\ldots,R_1\), apply Lemma \ref{lem:family_of_maps} with:
\begin{itemize}
\item \(\Pi=\Pi_i\),
\item \(x_j=x_{ij}\) for \(j=1,\ldots,r-s\), where \(x_{ij}\in C'\) are chosen in linearly general position with respect to \(\Pi_i\) and \(B_i\),
\item \(\bP^1=R_i\), and
\item \((y',y)=(y_{i-1},y_i)\).
\end{itemize}
Let \(\cF_i\) be the resulting family of maps \(h^i\colon R_i\to X\), and define \(\Pi_{i-1}=\Pi'\), \(B_{i-1}=B'\).

Choose \(x_{01},\ldots,x_{0,r-\ell+1}\in C'\) in linearly general position with respect to \(\Pi_0\) and \(B_0\). Let \(\Lambda_0=\langle x_{01},\ldots,x_{0,r-\ell+1}\rangle\cong \bP^{r-\ell}\) and let \(x=\Lambda_0\cap(\Pi_0\setminus B_0)\). Then, observe that \(x_{01},\ldots,x_{0,r-\ell+1},x\) must be in linearly general position on \(\Lambda_0\). Indeed, if this were not the case, then there would exist a point \(x\in\Pi_0\) lying on a common \((r-\ell-1)\)-plane with a strict subset of \(x_{01},\ldots,x_{0,r-\ell+1}\), contradicting linear generality. 

Therefore, there exists an \((r-\ell-1)\)-parameter family of embedded rational normal curves \(h^0\colon R_0\hookrightarrow\Lambda_0\) passing through all of the points \(x_{01},\ldots,x_{0,r-\ell+1}\). A general such rational normal curve \(h^0\) misses the proper linear subspace \(\bP^s\cap\Lambda_0\), because \(s\leq r-2\). Choose such a general \(h^0\), and define \(p_{01},\ldots,p_{0,r-\ell+1}\in R_0\) in such a way that \(h^0(p_{0j})=x_{0j}\) for \(j=1,\ldots,r-\ell+1\), in addition to \(h^0(y_0)=x\). Then, \(h^0\) lifts to a map \(h^0\colon R_0\hookrightarrow X\) of degree \(r-\ell,0\), as it image does not meet the blown-up subspace \(\bP^s\subset \bP^r\).

Finally, the point \(x\in\Pi_0\setminus B_0\) uniquely determines a map \(h^1\in\cF_1\), which in turn determines a point \(h^1(y_1)\in\Pi_1\setminus B_1\). Continuing in this fashion, we uniquely determine maps \([h^i\colon R_i\to X]\in \cF_i\), with the property that
\[
h^i(y_i)=h^{i+1}(y_i)\in\Pi_i\setminus B_i
\]
for each \(i\). In particular, we obtain a map \(h\colon R\to X\) satisfying the needed degree and incidence conditions.

It remains to prove that \(H^1(R,h^*T_X(-\sum p_{ij}))=0\). A straightforward calculation shows that \(\chi(R,h^*T_X(-\sum p_{ij}))=0\), so it suffices to show that \(H^0(R,h^*T_X(-\sum p_{ij}))=0\). This amounts to the statement that the map \(h\) admits no nontrivial first-order deformations, subject to the fixed incidence conditions \(h(p_{ij})=x_{ij}\).

Fix a first-order deformation of \(h\). We show first that the image \(h(y_0)\) deforms trivially to first order. Indeed, \(h(y_0)\) must deform along both \(\Lambda_0\) and \(\Pi_0\), but the intersection \(\Lambda_0\cap\Pi_0=x\) is transverse. Now, by Lemma \ref{lem:family_of_maps}(c), it follows that \(h^1\) also deforms trivially, and therefore that \(h(y_1)=h^1(y_1)\) also deforms trivially. Continuing inductively in this manner, we deduce that all of the restrictions \(h^1,\ldots,h^{\ell}\) deform trivially. Finally, the restriction \(h^0\colon R_0\to X\) deforms to first order, preserving \(r-\ell+2\) incidence conditions \(h^0(p_{i0})=x_{i0}\) and \(h(y_0)=x\). By Lemma \ref{lem:tangent_degenerate_curve} below, the deformation of \(h^0\) must also be trivial. Therefore, \(h\) admits no non-trivial deformations.

\end{proof}

\begin{lemma} \label{lem:tangent_degenerate_curve}
Let \(h_0\colon R_0=\bP^1\to X\) be a curve of degree \((r-\ell,0)\), where \(1\le \ell\le r-1\), given by the proper transform of a rational normal curve \(R_0\subset \Lambda_0\cong \bP^{r-\ell}\subset \bP^r\). Then, we have
\[
h_0^*T_X\cong \cO(r-\ell)^{\oplus \ell}\oplus \cO(r-\ell+1)^{\oplus r-\ell}.
\]
Furthermore, the subbundle \(\cO(r-\ell+1)^{\oplus r-\ell}\subset h_0^*T_X\) is identified with \(T_{\Lambda_0}|_{R_0}\).
\end{lemma}

\begin{proof}
We have a short exact sequence
\[
0\to T_{\Lambda_0}|_{R_0}\to h_0^*T_X\to N_{\Lambda_0/X}|_{R_0}\to 0.
\]
Arguing as in the proof of Lemma \ref{lem:tangent_r-s}, we have
\[
h_0^*T_X\cong N_{\Lambda_0/X}|_{R_0}\oplus T_{\Lambda_0}|_{R_0}\cong \cO(r-\ell)^{\oplus \ell}\oplus \cO(r-\ell+1)^{\oplus r-\ell},
\]
as needed.
\end{proof}

\subsection{\(T_\ell\): Conclusion} \label{sec:tree_conclusion}

\begin{proof}[Proof of Proposition \ref{prop:inductive_steps_main}\((T_\ell)\)]
First, assume that \(\ell\ge 1\), as is the case in Situation \ref{sit:curve_many_points}. Let \(f\colon C=C'\cup R\to X\) be the map given by gluing the maps \(f'\colon C'\to X\) and \(h\colon R\to X\) (Proposition \ref{prop:tree_exists}) along the nodes \(x_{ij}\). By the cohomology vanishing in of Proposition \ref{prop:tree_exists}, the criterion of Lemma \ref{lem:attach_and_smooth} applies, and we conclude immediately.

It remains to consider the case \(\ell=0\). In this case, we take \(x_1,\dots,x_{r+2}\in C'\) to be general points, and \(h\colon R\to \bP^r\) to be a rational normal curve in \(\bP^r\) passing through all of the \(x_i\). Arguing as in the proof of Proposition \ref{prop:tree_exists}, there is an \((r-1)\)-dimensional family of such rational normal curves \(h\), and a general such misses the blow-up center \(\bP^s\subset \bP^r\). Therefore, \(h\) lifts to a map \(h\colon R\to X\) of degree \((r,0)\). Because
\[
h^*T_X\cong h^*T_{\bP^r}\cong \cO(r+1)^{\oplus r},
\]
the criterion of Lemma \ref{lem:attach_and_smooth} again applies. This completes the proof.
\end{proof}

\section{Reductions} \label{sec:reductions}.

In this section, we use Proposition \ref{prop:inductive_steps_main} to reduce \(\Ex(g,d,k)\) to existence for a small set of \emph{terminal} valid tuples. We fix \(r,s\) throughout this section; for the given \(r,s\), the set of terminal tuples consists of the valid tuples where \(g=0\) and a finite set of exceptional tuples.

\begin{definition} \label{def:terminal}
We say that a valid tuple \((g,d,k)\) is \emph{terminal} if it lies in one of the following five families.
\begin{enumerate}
\item[(Ia)] \((g,d,k)=(r+1,2r,k)\), where
\[
s+1+\frac{1}{r-s}\le k\le s+1+ \frac{s+1}{r-s}
\]
\item[(Ib)] \((g,d,k)=(2r+2,3r,k)\), where
\[
s+1+\frac{s+2}{r-s} \le k\le s+1+\frac{s+1}{r-s-1}.
\]
\item[(II)] \((g,d,k)=\left((r-s-1)k-(r-s)(s+1)+(2r+2),\,(r-s-1)k-(r-s)(s+1)+3r,\,k\right)\), where
\[
\frac{(r-s)(s+1)}{r-s-1}\le k\le 2s+2.
\]
\item[(III)] \((g,d,k)=\left((r-s)(q+1),\,(r-s)(q+2)+(s-1),\,s+1+q\right)\), where
\[
q=\left\lceil \frac{s+1}{r-s}\right\rceil.
\]
\item[(IV)] \(g=0\).
\end{enumerate}
\end{definition}

\begin{remark}\label{rem:I_empty}
Because \(k\) is an integer, the lower bound on \(k\) in family (Ia) is equivalent to \(s+2\leq k\). Therefore, family (Ia) is empty unless
\[
\frac{s+1}{r-s}\geq 1 \Longleftrightarrow r\leq 2s+1.
\]
Similarly, family (Ib) is empty unless \(r\leq 2s+2\).
\end{remark}

\begin{proposition} \label{prop:reductions}
Let \((g,d,k)\) be a valid tuple. Then, either \((g,d,k)\) is terminal, or at least one of the operations (see Proposition \ref{prop:inductive_steps_main})
\begin{enumerate}
\item[(L)] \((g,d,k)\to (g-1,d-1,k)\)
\item[(R)] \((g,d,k)\to (g-(r-s),d-(r-s),k-1)\)
\item[\((T_\ell)\)] \((g,d,k)\to \left(g-(r+1)-\ell(r-s-1),d-r-\ell(r-s-1),k-\ell\right)\), where \(\ell\) is an integer satisfying \(0\le \ell\le s+1\)
\end{enumerate}
replaces \((g,d,k)\) with a valid tuple.
\end{proposition}

The strategy of proof of Proposition \ref{prop:reductions} is simple. Recall first from Remark \ref{rem:dimension_effects}, have the following effects on the profile \((\rho,\rho_\bullet,\rho')\).
\begin{itemize}
    \item[(L)] \((\rho,\rho_\bullet,\rho')\to(\rho-1,\rho_\bullet-1,\rho'-1)\).
    \item[(R)] \((\rho,\rho_\bullet,\rho')\to(\rho-(r-s),\rho_\bullet-1,\rho')\).
    \item[\((T_\ell)\)] \((\rho,\rho_\bullet,\rho')\to(\rho-(r-s-1)\ell,\rho_\bullet,\rho'-(s+1-\ell))\).
\end{itemize}
Given a valid tuple \((g,d,k)\) for which \(g,\rho,\rho_\bullet,\rho'>0\), we will show first (Lemma \ref{lem:apply_L}) that (L) can be applied, thereby decreasing \(\rho,\rho_\bullet,\rho'\) all by \(1\). Otherwise, we are either in terminal family (IV) of Definition \ref{def:terminal}, or one of \(\rho,\rho_\bullet,\rho'\) is zero. In these three cases, we will then show that an appropriate combination of decrements (R) and \((T_\ell)\) can be applied until reaching a terminal tuple in family (Ia)/(Ib), (II), (III), respectively.

\begin{lemma} \label{lem:apply_L}
Let \((g,d,k)\) be a valid tuple, and suppose that \(g,\rho,\rho_\bullet,\rho'>0\). Then, the operation (L) replaces \((g,d,k)\) with a valid tuple.
\end{lemma}

\begin{proof}
Clearly, \(g,k,\rho,\rho_\bullet,\rho'\) remain non-negative after applying (L).
\end{proof}

Therefore, to prove Proposition \ref{prop:reductions}, it suffices to consider the cases where \(g>0\) and one of \(\rho,\rho_\bullet,\rho'\) is \(0\). We consider the three cases separately, in what follows.

\subsection{\(\rho=0\) and families (Ia)/(Ib)} \label{sec:rho_0}

\begin{proposition} \label{prop:rho_0}
Let \((g,d,k)\) be a valid tuple with \(g>0\) and \(\rho=0\). Suppose that none of the operations (L), (R), \((T_\ell)\) replace \((g,d,k)\) with a valid tuple. Then, \((g,d,k)\) belongs to terminal family (Ia) or (Ib), Definition \ref{def:terminal}.
\end{proposition}

\begin{proof}
Because
\[
\rho=(r+1)d-rg-r(r+1)=0,
\]
we may write \(d=rn\) and \(g=(r+1)(n-1)\), for some integer \(n\ge 2\). In particular, we have \(g\ge r+1\). If \(\rho'\ge s+1\), then \((T_0)\) may be applied, which decreases \(g\) by \(r+1\), decreases \(\rho'\) by \(s+1\), and leaves \(\rho,\rho_\bullet,k\) constant. (Note that by Remark \ref{rmk:degree-checks}, the degrees $d$ and $d-k$ automatically remain non-negative.) Therefore, we assume henceforth that \(\rho'\le s\).

Now, we compute that
\[
\begin{aligned}
\rho'&=(r-s)(d-k)-(r-s-1)g-(r-s)(r-s-1) \\
&=(s+1)(n+r-s-1)-(r-s)k.
\end{aligned}
\]
Because \(0\le \rho'\le s\), we conclude that
\[
n=\left\lceil \frac{(r-s)k}{s+1}\right\rceil-(r-s-1).
\]

On the other hand, we have
\[
\rho_\bullet=\rho+(r-s)(s+1)-(r-s-1)k\ge 0,
\]
hence \(k\le \frac{(r-s)(s+1)}{r-s-1}\). Now,
\[
\begin{aligned}
n&=\left\lceil \frac{(r-s)k}{s+1}\right\rceil-(r-s-1) \\
&\le \left\lceil \frac{(r-s)^2}{r-s-1}\right\rceil-(r-s-1) \\
&=(r-s+2)-(r-s-1)=3.
\end{aligned}
\]
Therefore, we must have \(n=2\) or \(n=3\).

If \(n=2\), then \((g,d,k)=(r+1,2r,k)\). The required inequality
\[
s+1+\frac{1}{r-s}\le k\le s+1+\frac{s+1}{r-s}
\]
of terminal family (Ia) is equivalent to \(0\le \rho'\le s\).

If \(n=3\), then \((g,d,k)=(2r+2,3r,k)\). The required inequality
\[
s+1+\frac{s+2}{r-s} \le k\le s+1+ \frac{s+1}{r-s-1}
\]
of terminal family (Ib) is equivalent (on the left) to \(\rho'\le s\), and (on the right) to \(\rho_\bullet\ge 0\).
\end{proof}

We also record the profiles of terminal families (Ia) and (Ib).

\begin{lemma} \label{lem:rho_0_profile}
In terminal families (Ia) and (Ib), respectively, the profiles \((\rho,\rho_\bullet,\rho')\) are
\begin{enumerate}
\item[(a)] \(\left(0,(r-s)(s+1)-(r-s-1)k,(s+1)(r-s+1)-(r-s)k\right)\),
\item[(b)] \(\left(0,(r-s)(s+1)-(r-s-1)k,(s+1)(r-s+2)-(r-s)k\right)\).
\end{enumerate}
In both cases, we have \(0\le \rho'\le s\).
\end{lemma}

\begin{remark}\label{rmk: rho_b=0-iff-rho'=0}
        When $s=r-2$, in the terminal family (Ib) one has $\rho_\bullet=0$ if and only if $\rho'=0$.
\end{remark}

\subsection{\(\rho_\bullet=0\) and family (II)} \label{sec:rho_bullet_0}

\begin{proposition} \label{prop:rho_bullet_0}
Let \((g,d,k)\) be a valid tuple with \(g>0\) and \(\rho_\bullet=0\). Suppose that none of the operations (L), (R), \((T_\ell)\) replace \((g,d,k)\) with a valid tuple. Then, \((g,d,k)\) belongs to terminal family (II), Definition \ref{def:terminal}.
\end{proposition}

\begin{lemma} \label{lem:g-d+r}
In the setting of Proposition \ref{prop:rho_bullet_0}, we have \(g-d+r=2\).
\end{lemma}

\begin{proof}
First, we have
\[
\rho_\bullet=0\leftrightarrow \rho=(r-s-1)k-(r-s)(s+1).
\]
Moreover, we compute that
\[
\begin{aligned}
\rho'&=\rho'-\rho_\bullet \\
&=(s+1)(g-d+r)-k.
\end{aligned}
\]
Now, \(\rho\ge 0\) implies that \(k\ge s+2\). Therefore,
\[
\rho'\ge 0\implies g-d+r\ge \frac{k}{s+1}\implies g-d+r\ge 2.
\]

We next prove that, if \(g-d+r\ge 3\), then \((T_\ell)\) may be applied, for some \(\ell\). This will imply the Lemma.

\underline{Case 1}: \(\rho'\ge s+1\). We claim that \((T_0)\) can be applied. Indeed, the values of \(\rho,\rho_\bullet,k\) stay the same, and \(\rho'\) decreases by \(s+1\). The value of \(g\) will decrease by \(r+1\), so we only need to check that \(g\ge r+1\). However, we have
\[
g=\rho+(r+1)(g-d+r)>r+1,
\]
so \((T_0)\) replaces \((g,d,k)\) with a valid tuple. 

\underline{Case 2}: \(0\le \rho'\le s\). We claim that \((T_\ell)\) replaces \((g,d,k)\) with a valid tuple, where \(\ell=s+1-\rho'\). \((T_\ell)\) has no effect on \(\rho_\bullet\) and decreases \(\rho'\) to 0. It suffices to check that the quantities
\[
k-\ell,\quad \rho-(r-s-1)\ell,\quad g-(r+1)-(r-s-1)\ell
\]
are all non-negative.
First, we have
\[
\begin{aligned}
k-\ell
&=\left[(s+1)(g-d+r)-\rho'\right]-\left[(s+1)-\rho'\right] \\
&=(s+1)(g-d+r-1)\ge 0.
\end{aligned}
\]
Next,
\[
\begin{aligned}
\rho-(r-s-1)\ell
&=(r-s-1)(k-\ell)-(r-s)(s+1) \\
&=(r-s-1)(s+1)(g-d+r-1)-(r-s)(s+1) \\
&=(s+1)\left[(r-s-1)(g-d+r-1)-(r-s)\right] \\
&\ge (s+1)\left[(r-s-1)2-(r-s)\right]\ge 0.
\end{aligned}
\]
Finally, we have
\[
\begin{aligned}
g-(r+1)-(r-s-1)\ell
&=(g-\rho)-(r+1)-\left[\rho-(r-s-1)\ell\right] \\
&\ge (r+1)(g-d+r-1)\ge 0.
\end{aligned}
\]
\end{proof}

\begin{proof}[Proof of Proposition \ref{prop:rho_bullet_0}]
By Lemma \ref{lem:g-d+r}, we have \(g-d+r=2\). Therefore, we have
\[
\begin{aligned}
g&=\rho+(r+1)(g-d+r)=\rho+2r+2,\\
d&=\rho+3r,
\end{aligned}
\]
where \(\rho=(r-s-1)k-(r-s)(s+1)\). The required inequality
\[
\frac{(r-s)(s+1)}{r-s-1}\le k\le 2s+2
\]
of terminal family (II) is equivalent (on the left) to \(\rho\ge 0\), and (on the right) to \(\rho'\ge 0\).
\end{proof}

\begin{lemma} \label{lem:rho_bullet_0_profile}
In terminal family (II), the profile \((\rho,\rho_\bullet,\rho')\) is
\[
\left((r-s-1)k-(r-s)(s+1),\,0,\,2(s+1)-k\right).
\]
Moreover, we have \(0\le \rho'\le s\).
\end{lemma}

\begin{proof}
The formulas for \(\rho,\rho'\) follow from the calculations in the proof of Lemma \ref{lem:g-d+r}. The inequality \(0\le \rho'\le s\) follows from the bounds on \(k\) required in terminal family (II).
\end{proof}

\subsection{\(\rho'=0\) and family (III)} \label{sec:rho_prime_0}

\begin{proposition} \label{prop:rho_prime_0}
Let \((g,d,k)\) be a valid tuple with \(g>0\) and \(\rho'=0\). Assume furthermore that \(\rho,\rho_\bullet>0\). Suppose that none of the operations (L), (R), \((T_\ell)\) replace \((g,d,k)\) with a valid tuple. Then, \((g,d,k)\) belongs to terminal family (III), Definition \ref{def:terminal}.
\end{proposition}

Note that it is safe to assume that \(\rho,\rho_\bullet>0\), because otherwise, one can instead apply Proposition \ref{prop:rho_0} or \ref{prop:rho_bullet_0}. Because
\[
\rho'=(r-s)(d-k)-(r-s-1)g-(r-s)(r-s-1)=0,
\]
we may write \(g=(r-s)(n-1)\) and \(d-k=(r-s-1)n\), for some integer \(n\ge 2\).

\begin{lemma} \label{lem:n_k_minus_s}
With notation as above, we must have \(n=k-s+1\).
\end{lemma}

\begin{proof}
We claim first that, if \(\rho\ge r-s\), then (R) can be applied. Indeed, \(\rho,\rho_\bullet,\rho'\) decrease by \(r-s,1,0\), respectively, so all remain non-negative. The value of \(g=(r-s)(n-1)\ge r-s\) decreases by \(r-s\), and \(k\) decreases by \(1\). We only need to check that \(k\geq 0\). Indeed, if \(k=0\), then
\[
\begin{aligned}
\rho_\bullet
&= \rho'-(s+1)(g-d+r)+k \\
&= (s+1)(-s-n) \\
&<0,
\end{aligned}
\]
contradiction. Therefore, we assume henceforth that \(\rho\le r-s-1\).

Now, we have
\[
\begin{aligned}
\rho&=(r+1)\left[k+(r-s-1)n\right]-r(r-s)(n-1)-r(r+1) \\
&= (r+1)k-(s+1)n-r(s+1) \\
&= (r-s)k+(s+1)(k-n)-r(s+1).
\end{aligned}
\]

Because \(0\leq \rho\leq r-s-1\), it follows that
\[
k=\left\lceil \frac{(s+1)(r-(k-n))}{r-s}\right\rceil.
\]
In particular, the value of \(k-n\) determines the value of \(k\). Furthermore, because \(n\geq 2\), we have

\[
\begin{alignedat}{2}
\phantom{\Longleftrightarrow}&\;&\left\lceil \frac{(s+1)(r-(k-n))}{r-s}\right\rceil&\geq (k-n)+2 \\
\Longleftrightarrow&\;&(s+1)(r-(k-n))&>((k-n)+1)(r-s) \\
\Longleftrightarrow&\;&k-n&<s.
\end{alignedat}
\]

Next, we compute
\[
\begin{aligned}
\rho_\bullet
&=(s+1)(g-d+r)+k \\
&=(s+1)(-s-n+k)+k \\
&=(s+2)k-(s+1)n-s(s+1) \\
&=k+(s+1)(k-n)-s(s+1) \\
&=\left\lceil \frac{(s+1)(r-(k-n))}{r-s}\right\rceil
 +(s+1)((k-n)-s)>0.
\end{aligned}
\]

Therefore, we have
\[
\begin{aligned}
(s+1)(r-(k-n)) &> (r-s)(s+1)(s-(k-n)) \\
\Longleftrightarrow [(r-s)-1][(s-(k-n))-1] &< 1.
\end{aligned}
\]
Because \(r-s\geq 2\) and \(k-n<s\), it follows that \(k-n=s-1\), as needed.
\end{proof}

\begin{proof}[Proof of Proposition \ref{prop:rho_prime_0}]
Using that \(n=k-s+1\), we have
\[
\rho=(r-s)k-(s+1)(r-s+1).
\]
From the proof of Lemma \ref{lem:n_k_minus_s}, we have \(0\le \rho\le r-s-1\), which implies that
\[
k=\left\lceil\frac{(s+1)(r-s+1)}{r-s}\right\rceil=s+1+\left\lceil\frac{s+1}{r-s}\right\rceil=:s+1+q.
\]
From here, using that \(q=n-2\), we deduce that \((g,d,k)\) lies in terminal family (III).
\end{proof}

The proof of Proposition \ref{prop:rho_prime_0} also gives the following.

\begin{lemma} \label{lem:rho_prime_0_profile}
In terminal family (III), the profile \((\rho,\rho_\bullet,\rho')\) is
\[
\left((r-s)q-(s+1),\,q,\,0\right).
\]
Moreover, we have \(0\le \rho\le r-s-1\) and \(k=s+1+q\).
\end{lemma}

\section{Base cases}
\label{sec:base_cases}

In this section, we handle existence for the terminal families of Definition \ref{def:terminal}, and deduce our main existence theorem, Theorem \ref{thm:main}. In order to streamline the arguments, we will assume in \S\ref{sec:terminal_base} that \(\charac(\bK)=0\). However, we will explain afterward in \S\ref{sec:positive_char} how to circumvent this hypothesis.

\subsection{Genus 0}
\label{sec:genus_0}

\begin{proposition}
\label{prop:genus_0}
Let \((g=0,d,k;r,s)\) be a valid tuple. Then, \(\Ex(0,d,k;r,s)\) holds.
\end{proposition}

\begin{proof}
Recall that \(\rho,\rho'\ge 0\) imply that \(d\ge r\) and \(d-k\ge r-s-1\), respectively. Therefore, we may choose:
\begin{itemize}
\item a nonzero \(f_0\in H^0(\bP^1,\cO(k))\),
\item linearly independent \(f_1,\dots,f_{r-s}\in H^0(\bP^1,\cO(d-k))\), with no common vanishing point, and
\item \(f_{r-s+1},\dots,f_{r+1}\in H^0(\bP^1,\cO(d))\), such that
\[
f_0f_1,\dots,f_0f_{r-s},f_{r-s+1},\dots,f_{r+1}
\]
are linearly independent, with no common vanishing point.
\end{itemize}
Then,
\[
f=[f_0f_1:\cdots:f_0f_{r-s}:f_{r-s+1}:\cdots:f_{r+1}]
\]
defines a map \(f\colon \bP^1\to X_{r,s}\) of degree \((d,k)\). By the linear independence assumptions, \(f\) is non-degenerate, and by \cite[Lemma 2.1.2]{cl_interpolation}, \(f\) is nice.
\end{proof}

\subsection{Serre duality}
\label{sec:serre}

For the remaining terminal families (Ia)/(Ib), (II), (III), the basic mechanism by which we will prove existence is the following.
\begin{enumerate}
\item[(1)] Produce a smooth point
\[
\uw\in G^{r,r-s-1}_{d,k}(C)
\]
from the data of an ILS with different invariants \((g,\dtilde,k,\rtilde,\ws)\), via Serre duality.
\item[(2)] Show that \(\uw\) is bpf, and therefore underlies a nice map.
\end{enumerate}

Serre duality provides a mechanism to produce a \emph{complete} linear series \(W=H^0(\cL)\) from the data of the complete linear series \(H^0(K_C\otimes\cL^\vee)\). Therefore, we will need the following notion.

\begin{definition}
\label{def:complete}
Let
\[
\uw=(\cL,D,V\subset W)\in G^{r,r-s-1}_{d,k}(C)
\]
be an ILS. We say that \(\uw\) is \emph{complete} if both \(W=H^0(\cL)\) and \(V=H^0(\cL(-D))\).
\end{definition}

It will be important to allow \(s=r-1\) when considering moduli spaces of ILS \(G^{r,r-s-1}_{d,k}(C)\).

\begin{notation}
\label{notation:serre}
Let
\[
\uw=(\cL,D,V\subset W)\in G^{r,r-s-1}_{d,k}(C)
\]
be an ILS. Write
\[
\begin{aligned}
\cM &= K_C\otimes\cL(-D)^\vee,\\
\cM(-D) &= K_C\otimes\cL^\vee.
\end{aligned}
\]
Then, \(\cM\) is a line bundle on \(C\) of degree
\[
\dtilde:=(2g-2)-d+k,
\]
and \(\deg(\cM(-D))=\dtilde-k\). If \(\uw\) is complete, then
\[
\begin{aligned}
h^0(\cM) &= \rtilde+1,\\
h^0(\cM(-D)) &= \rtilde-\ws,
\end{aligned}
\]
where
\[
\begin{aligned}
\rtilde &= g-d+k+r-s-2,\\
\ws &= k-s-2.
\end{aligned}
\]
Even if \(\uw\) is not complete, we take the above expressions as definitions for the quantities \(\rtilde,\ws\). 
\end{notation}

\begin{remark}
    Note that, if the profile of \((g,d,k;r,s)\) is \((\rho,\rho_\bullet,\rho')\), then the profile of the Serre dual tuple \((g,\dtilde,k;\rtilde,\ws)\) is \((\rho',\rho_\bullet,\rho)\).
\end{remark}

\begin{lemma}
\label{lem:serre}
Suppose that
\[
\uw=(\cL,D,V\subset W)\in G^{r,r-s-1}_{d,k}(C)
\]
is complete, or equivalently that
\[
\begin{aligned}
h^0(\cM) &= \rtilde+1,\\
\text{and}\qquad h^0(\cM(-D)) &= \rtilde-\ws.
\end{aligned}
\]
Then, Serre duality induces a rational map
\[
s\colon G^{r,r-s-1}_{d,k}(C)
\dashrightarrow
G^{\rtilde,\rtilde-\ws-1}_{\dtilde,k}(C)
\]
that is an isomorphism in a neighborhood of \(\uw\). 

In particular, if \(G^{\rtilde,\rtilde-\ws-1}_{\dtilde,k}(C)\) has a smooth, complete point of expected dimension, then so does \(G^{r,r-s-1}_{d,k}(C)\).
\end{lemma}

\begin{proof}
That \(s\) is birational is immediate from the fact that completeness is an open condition in \(G^{r,r-s-1}_{d,k}(C)\). The equivalence of existence of smooth, complete points of expected dimension follows from the fact that both moduli spaces have expected dimension \(\rho_\bullet\).
\end{proof}

Given $\widetilde{\underline{W}} \in G^{\widetilde{r},\widetilde{r}-\widetilde{s}-1}_{\widetilde{d},k}(C)$, in order to deduce \(\Ex(g,d,k;r,s)\), one needs to know in addition that the constructed smooth point \(\underline{W}\in G^{r,r-s-1}_{d,k}(C)\) of expected dimension is bpf. In practice, the simplest way to guarantee this is the following. 

\begin{lemma}\label{lem:check_bpf}
Let \(C\) be a general curve of genus \(g\), and fix \(r,s,d,k\) with $s\le r-2$. Assume that, for all \(d',k'\) satisfying $d'\le d$ and $d'-k'\le d-k$, the moduli space of ILS \(G^{r,r-s-1}_{d',k'}(C)\) is pure of expected dimension. Then, the bpf locus in \(G^{r,r-s-1}_{d,k}(C)\) is dense.
\end{lemma}

\begin{proof}
The proof is very similar to (but easier than) that of \cite[Proposition 2.6.2]{cl_complete}. Suppose for sake of contradiction that there is an irreducible component \(G\subset G^{r,r-s-1}_{d,k}(C)\), on which a general point \(\uw\in G\) is not bpf. Fix such a general point \(\uw=(\cL,D,V\subset W)\). Consider the following two operations on ILS (\cite[Definition 2.5.1]{cl_complete}).
 \begin{enumerate}
    \item[(T1)] If every section of $V\subset H^0(C,\cL(-D))$ vanishes at $p\in C$, then replace $D$ with $D+p$, leaving $\cL,V,W$ the same. We now have $(\cL,D,V\subset W)\in G^{r,r-s-1}_{d,k+1}(C)$.
    
    \item[(T2)] If every section of $W$ vanishes at $p\in \Supp(D)$, then replace $D$ with $D-p$ and $\cL$ with $\cL(-p)$. The subspaces $V\subset W$ are identified with subspaces of the new $H^0(C,\cL)$ in the obvious way. We now have $(\cL,D,V\subset W)\in G^{r,r-s-1}_{d-1,k-1}(C)$.   
\end{enumerate}

We may apply the operation (T1) a total of \(\delta_1\geq 0\) times, at various points \(p\in C\), until property (BPF1) is satisfied, and then apply operation (T2) \(\delta_2\geq 0\) times, until \(\uw\) is bpf. By assumption, \(\delta_1,\delta_2\) are not both \(0\). After possibly passing to an étale neighborhood of \(G\)\footnote{In order to make sense of the operations (T1), (T2) in a neighborhood of \(\uw\), one needs a consistent choice of point \(p\in C\) over which the twisting can be applied. Such a choice may only exist after passing to an étale open neighborhood of \(\uw\); see also the footnote to \cite[Proof of Proposition 2.6.2]{cl_complete}.}, this sequence of operations defines a rational map
\[
\tau\colon G\dashrightarrow G^{r,r-s-1}_{d-\delta_2,k+\delta_1-\delta_2}(C),
\]
dominating the bpf locus. The fibers of \(\tau\) have dimension \(\delta_2\), corresponding to choices of twisted points \(p\in C\) in the operations (T2). Therefore, we have
\[
\begin{aligned}
\dim(G)&\le\delta_2+\dim\bigl(G^{r,r-s-1}_{d-\delta_2,k+\delta_1-\delta_2}(C)\bigr) \\
&=\rho_\bullet-\delta_1(r-s-1)-\delta_2(s+1) \\
&<\rho_\bullet,
\end{aligned}
\]
which is a contradiction.
\end{proof}

\begin{corollary}\label{cor:char_0_serre}
Suppose that \(\charac(\bK)=0\), and that \(G^{\rtilde,\rtilde-\ws-1}_{\dtilde,k}(C)\) contains a smooth, complete point of expected dimension. Then, \(\Ex(g,d,k;r,s)\) holds.
\end{corollary}

\begin{proof}
By Lemma \ref{lem:serre}, there exists a non-empty open subset \(U\subset G^{r,r-s-1}_{d,k}(C)\) in which a general point is smooth of expected dimension. If \(\charac(\bK)=0\), then by Theorem \ref{thm:farkas}, the hypothesis of Lemma \ref{lem:check_bpf} holds. In particular, a general point of \(U\) is also bpf, so gives rise to a nice map.
\end{proof}

\subsection{Terminal families}\label{sec:terminal_base}

In order to streamline the remaining arguments, we will assume for the remainder of this section that \(\charac(\bK)=0\), so that we may apply Corollary \ref{cor:char_0_serre}. We explain in \S\ref{sec:positive_char} how to remove this hypothesis.

\begin{proposition}\label{prop:serre_Ia}
Let \((g,d,k;r,s)\) be a terminal tuple of type (Ia), Definition \ref{def:terminal}. Then, \(\Ex(g,d,k;r,s)\) holds.
\end{proposition}

\begin{proof}
The Serre dual tuple is
\[
(g,\dtilde,k;\rtilde,\ws)
=
(r+1,k,k;k-s-1,k-s-2).
\]
Note in particular that \(\dtilde=k\) and that \(\ws=\rtilde-1\). Therefore, the ILS space \(G^{\rtilde,\rtilde-\ws-1}_{\dtilde,k}(C)=G^{\rtilde,0}_{\dtilde,\dtilde}(C)\) is identified with the projectivization of the universal linear series over the usual moduli space of linear series \(G^{\rtilde}_{\dtilde}(C)\). In particular, by the Brill-Noether existence and Gieseker-Petri theorems, \(G^{\rtilde,\rtilde-\ws-1}_{\dtilde,k}(C)\) is non-empty and smooth of dimension
\[
\rho(g,\rtilde,\dtilde)+\rtilde
=
\rho'+(k-s-1)
=
\rho_\bullet,
\]
where we have applied Lemma \ref{lem:rho_0_profile}(a).

Therefore, by Corollary \ref{cor:char_0_serre}, it suffices to show that \(G^{\rtilde,\rtilde-\ws-1}_{\dtilde,k}(C)\) contains a \emph{complete} point. To see this, note that \((g,\rtilde,\dtilde)\) lies in the ``Brill-Noether special'' range:
\[
g-\dtilde+\rtilde=r-s\geq 0.
\]
Therefore, by \cite[Lemma IV.3.5]{acgh}, a general line bundle \(\cM\) on \(C\) of degree \(k\) with \(h^0(C,\cM)\geq \rtilde+1\) has in fact \(h^0(C,\cM)=\rtilde+1\). In particular, a general point of \(G^{\rtilde,\rtilde-\ws-1}_{\dtilde,k}(C)\) has \(h^0(C,\cM)=\rtilde+1\), as well as
\[
h^0(C,\cM(-D))=h^0(C,\cO_C)=1.
\]
so is complete.
\end{proof}

\begin{proposition}\label{prop:serre_III}
Let \((g,d,k;r,s)\) be a terminal tuple of type (III), Definition \ref{def:terminal}. Then, \(\Ex(g,d,k;r,s)\) holds.
\end{proposition}

\begin{proof}
Recalling that we set \(q=\left\lceil \frac{s+1}{r-s}\right\rceil\), the Serre dual tuple is
\[
(g,\dtilde,k;\rtilde,\ws)
=
\bigl(g,(r-s+1)q,s+1+q;q,q-1\bigr),
\]
where we have also applied Lemma \ref{lem:rho_prime_0_profile}. Note in particular that \(\ws=\rtilde-1\).

Therefore, the moduli space of ILS \(G^{\rtilde,\rtilde-\ws-1}_{\dtilde,k}(C)=G^{\rtilde,0}_{\dtilde,k}(C)\) parameterizes the data of:
\begin{itemize}
    \item a point \(\widetilde{W}\in G^{\rtilde}_{\dtilde}(C)\),
    \item an effective divisor \(D\subset C\) of degree \(k=s+1+q\), and
    \item a line \(\ell\subset \widetilde{W}\), consisting of sections vanishing along \(D\).
\end{itemize}
We have
\[
\dim\bigl(G^{\rtilde}_{\dtilde}(C)\bigr)=\rho(g,\rtilde,\dtilde)=\rho'=0.
\]
Therefore, \(G^{\rtilde}_{\dtilde}(C)\) consists of finitely many smooth points, of which we may choose one, \(\widetilde{W}\in G^{\rtilde}_{\dtilde}(C)\). Then, \(\widetilde{W}\) is bpf, because \(\rho(g,\rtilde,\dtilde-1)<\rho(g,\rtilde,\dtilde)=0\), and complete, because \(\rho(g,\rtilde+1,\dtilde)<\rho(g,\rtilde,\dtilde)=0\). Let \(f\colon C\to \bP^{\rtilde}\) be a map with underlying linear series \(\widetilde{W}\). 

The data of a line \(\ell\subset \widetilde{W}\) of sections vanishing along \(D\) is equivalently the data of a hyperplane \(H\subset \bP^{\rtilde}\) for which \(D\subset f^{-1}(H)\). Therefore, the open and closed subscheme \(G_{\widetilde{W}}\subset G^{\rtilde,\rtilde-\ws-1}_{\dtilde,k}(C)\), consisting of ILS with underlying \(\widetilde{W}\in G^{\rtilde}_{\dtilde}(C)\), admits the structure of a finite cover
\[
\gamma\colon G_{\widetilde{W}}\to (\bP^{\rtilde})^\vee
\]
of the space of hyperplanes \(H\subset \bP^{\rtilde}\). In particular, \(G_{\widetilde{W}}\) is pure of dimension \(\rtilde=q=\rho_\bullet\), the expected. The last equality used Lemma \ref{lem:rho_prime_0_profile}.

By Bertini smoothness, a general hyperplane \(H\subset \bP^{\rtilde}\) has the property that \(f^{-1}(H)\) is reduced, which implies that every point of \(\gamma^{-1}(H)\) is a smooth point of \(G_{\widetilde{W}}\). In particular, any general point of \(G_{\widetilde{W}}\) is smooth of expected dimension.

Finally, it suffices to show that a general point of \(G_{\widetilde{W}}\) is also complete. This amounts to the statement that \(H\subset \bP^{\rtilde}\) is the \emph{unique} hyperplane containing \(D\). Because
\[
k=s+1+q>q=\rtilde,
\]
one can arrange for \(D\) to contain a collection of \(\rtilde\) general points on \(C\), which in particular lie on a unique hyperplane. Therefore, we conclude by Corollary \ref{cor:char_0_serre}.
\end{proof}

For the remaining terminal tuples, of types (Ib) and (II), we do not obtain unconditional proofs of existence. Rather, we reduce existence to that of the Serre dual tuple, which has \(\ws=\rtilde-2\).

\begin{proposition}\label{prop:serre_Ib}
Let \((g,d,k;r,s)\) be a terminal tuple of type (Ib), Definition \ref{def:terminal}. Let
\[
(g,\dtilde,k;\rtilde,\ws)
=
(2r+2,r+2+k,k;k-s,k-s-2)
\]
be the Serre dual tuple. If \(\Ex(g,\dtilde,k;\rtilde,\ws)\) holds, then so does \(\Ex(g,d,k;r,s)\).
\end{proposition}

\begin{proof}
By Corollary \ref{cor:char_0_serre}, it suffices to show that any point \((\cM,D,\widetilde{V}\subset \widetilde{W})\in G^{\rtilde,\rtilde-\ws-1}_{\dtilde,k}(C)=G^{\rtilde,1}_{\dtilde,k}(C)\) is complete. Note first that \(\widetilde{V}\in G^1_{\dtilde-k}(C)\), which has dimension \(\rho'=0\). Therefore, we must have \(\widetilde{V}=H^0(C,\cM(-D))\).

Suppose now, for sake of contradiction, that \(h^0(C,\cM)>\rtilde+1\), so that \(\widetilde{W}\subsetneq H^0(C,\cM)\). Then, there is a family of dimension at least \(\rtilde-1\) of subspaces \(\widetilde{W}\subset H^0(C,\cM)\) containing \(\widetilde{V}\). In particular, we must have
\[
\rtilde-1\leq \dim\bigl(G^{\rtilde,\rtilde-\ws-1}_{\dtilde,k}(C)\bigr)=\rho_\bullet=(r-s)(s+1)-(r-s-1)k,
\]
by Theorem \ref{thm:farkas}. However, substituting \(\rtilde=k-s\) and rearranging yields \(k\leq s+1+\frac{s+1}{r-s}\), which contradicts the inequality on \(k\) required in family (Ib).
\end{proof}

\begin{proposition}\label{prop:serre_II}
Let \((g,d,k;r,s)\) be a terminal tuple of type (II), Definition \ref{def:terminal}. Let
\[
(g,\dtilde,k;\rtilde,\ws)=(\rho+2r+2,\rho+r+k+2,k;k-s,k-s-2)
\]
be the Serre dual tuple. If \(\Ex(g,\dtilde,k;\rtilde,\ws)\) holds, then so does \(\Ex(g,d,k;r,s)\).
\end{proposition}

\begin{proof}
By Corollary \ref{cor:char_0_serre}, it suffices to show that any point \((\cM,D,\widetilde{V}\subset \widetilde{W})\in G^{\rtilde,\rtilde-\ws-1}_{\dtilde,k}(C)=G^{\rtilde,1}_{\dtilde,k}\) is complete. If this were not the case, then it would be possible to keep \(\cM,D\) the same, while either enlarging \(\widetilde{W}\) and keeping \(\widetilde{V}\) the same, or enlarging \(\widetilde{V}\) and keeping \(\widetilde{W}\) the same. However, recall that
\[
\rho_\bullet=g-(\rtilde+1)(g-\dtilde+\rtilde)+(\rtilde-\ws)(\ws+1)-(\rtilde-\ws-1)k=0.
\]

Enlarging \(\widetilde{W}\) and keeping \(\widetilde{V}\) the same produces an ILS for which the values of \(\rtilde,\ws\) both increase by \(1\). In this case, the new value of \(\rho_\bullet\) is
\[
\begin{aligned}
\dtilde-g-2(\rtilde+1)+(\rtilde-\ws)
&=(\rho+r+k+2)-(\rho+2r+2)-2(k-s+1)+2 \\
&=(s-r)+(s-k) \\
&<0,
\end{aligned}
\]
where in the last inequality we have used 
\[
k \geq \frac{(r-s)(s+1)}{r-s-1}=(s+1)\left(1+ \frac{1}{r-s-1}\right)>s+1
\]
contradicting Theorem \ref{thm:farkas}.

On the other hand, enlarging \(\widetilde{V}\) and keeping \(\widetilde{W}\) the same produces an ILS for which \(\ws\) decreases by \(1\). In this case, the new value of \(\rho_\bullet\) is \(2\widetilde{s}-\widetilde{r}-k=-s-4<0\), again contradicting Theorem \ref{thm:farkas}.
\end{proof}

If \(s=r-2\) and \((g,d,k;r,s)\) is a terminal tuple of type (II), then we must have \(k=2s+2\). In this case, we have \((g,\dtilde,k;\rtilde,\ws)=(g,d,k;r,s)\), so Proposition \ref{prop:serre_II} is of no use. However we have a direct argument.

\begin{proposition}\label{prop:all_zeros}
Suppose that \(s=r-2\), and let \((g,d,k;r,s)\) be a terminal tuple of type (II), in which case
\[
(g,d,k;r,s)=(2r+2,3r,2r-2;r,r-2),\qquad (\rho,\rho_\bullet,\rho')=(0,0,0).
\]
Then, \(\Ex(g,d,k;r,s)\) holds.
\end{proposition}

\begin{proof}
We will construct an ILS \((\cL,D,V\subset W)\in G^{r,1}_{3r,2r-2}(C)\). Observe first that we must have \(V\in G^1_{r+2}(C)\) and \(W\in G^r_{3r}(C)\), both of which are smooth of dimension \(\rho'=\rho=0\). Furthermore, writing \(\cL\) for the line bundle underlying \(W\), we have
\[
h^0(K_C\otimes \cL^{-1})=h^0(\cL)-(3r-(2r+2)+1)=2,
\]
because $\rho=0$ implies $W=H^0(C,\cL)$. In particular, \(W\) must be residual to a complete \(g^1_{r+2}\).

Choose now distinct line bundles \(\cN_1,\cN_2\in W^1_{r+2}(C)\), that is, \(\deg(\cN_i)=r+2\) and \(h^0(\cN_i)=2\). This is possible because the cardinality of \(W^1_{r+2}(C)\) is equal to the Catalan number \(C_{r+1}=\frac{1}{r+2}\binom{2r+2}{r+1}\geq 2\) \cite{castelnuovo,gh}. Then, define \(\cL=K_C\otimes \cN_1^{-1}\), and \(W=H^0(\cL)\).

We claim next that
\[
h^0(\cL\otimes \cN_2^{-1})=h^0(K_C\otimes \cN_1^{-1}\otimes \cN_2^{-1})\geq 1.
\]
Equivalently, by Serre duality, \(h^0(\cN_1\otimes \cN_2)\geq 4\). Indeed, by the base-point-free pencil trick \cite[\S III.3]{acgh}, the multiplication map \(H^0(\cN_1)\otimes H^0(\cN_2)\to H^0(\cN_1\otimes \cN_2)\) is injective if \(\cN_1\not\simeq \cN_2\), hence \(h^0(\cN_1\otimes \cN_2)\geq 4\).

Choose now \(D\in |\cL\otimes \cN_2^{-1}|\), so that \(V=H^0(\cL(-D))\in G^1_{r+2}(C)\). Then, \((\cL,D,V\subset W)\in G^{r,1}_{3r,2r-2}(C)\). If it were the case that \(h^0(\cL\otimes \cN_2^{-1})\geq 2\), then we would obtain a positive-dimensional family inside \(G^{r,1}_{3r,2r-2}(C)\), which contradicts Theorem \ref{thm:farkas}. Therefore, we must in fact have a \emph{unique} choice of \(D\). We claim finally that \((\cL,D,V\subset W)\) is a smooth point of \(G^{r,1}_{3r,2r-2}(C)\). Indeed, \(\cN_1,\cN_2\) do not deform to first order inside \(W^1_{r+2}(C)\) by the Gieseker-Petri theorem, so \(V,W\) do not deform to first order. Because \(h^0(\cL\otimes \cN_2^{-1})=1\), the divisor \(D\) also does not deform to first order. 

Combining Lemma \ref{lem:check_bpf} and Theorem \ref{thm:farkas}, the ILS \((\cL,D,V\subset W)\) is also bpf, so gives rise to a nice map. This completes the proof.
\end{proof}

\begin{remark}\label{rem:hurwitz_stronger}
One can obtain a slightly more general existence statement than Proposition \ref{prop:all_zeros} using Hurwitz-Brill-Noether theory, see Proposition \ref{prop:hurwitz} and its proof.
\end{remark}

We are now ready to prove our main existence theorem.

\begin{proof}[Proof of Theorem \ref{thm:main}, assuming \(\charac(\bK)=0\)]
We begin with the case \(s=r-2\). By Propositions \ref{prop:inductive_steps_main} and \ref{prop:reductions}, we reduce immediately to the case of terminal valid tuples, Definition \ref{def:terminal}. In families (Ia), (II), (III), and (IV), existence is proven unconditionally in Propositions \ref{prop:serre_Ia}, \ref{prop:all_zeros}, \ref{prop:serre_III}, and \ref{prop:genus_0}, respectively.

Consider now a valid tuple \((g,d,k;r,s)\) in terminal family (Ib), with profile \((0,\rho_\bullet,\rho')\). By Lemma \ref{lem:rho_0_profile}(b) and Remark \ref{rmk: rho_b=0-iff-rho'=0}, we have \(\rho_\bullet=0\) if and only if \(\rho'=0\). If \(\rho_\bullet=\rho'=0\), then we are in terminal family (II), in which case we are done. Thus, we may assume that \(\rho_\bullet,\rho'>0\). In this case, Proposition \ref{prop:serre_Ib} reduces existence to that for the Serre dual tuple \((g,\dtilde,k;\rtilde,\ws=\rtilde-2)\), with profile \((\rho'>0,\rho_\bullet>0,0)\). By Proposition \ref{prop:rho_prime_0}, \(\Ex(g,\dtilde,k;\rtilde,\ws)\) can be reduced further to existence for a valid tuple in terminal family (III), in which case we are done. This completes the proof of existence in the case \(s=r-2\).

Now, assume that \(r,s\) are arbitrary. Once more, we reduce existence to the case of terminal valid tuples. Now, all base cases are handled either unconditionally, or reduced to the case \(\ws=\rtilde-2\) in Propositions \ref{prop:serre_Ib} and \ref{prop:serre_II}, which follow in turn from the case \(s=r-2\) above. This completes the proof.
\end{proof}

\section{Positive characteristic}
\label{sec:positive_char}

In this section, we explain how to modify the arguments of \S \ref{sec:terminal_base} to prove Theorem \ref{thm:main} in arbitrary characteristic. For the most part, we only used the characteristic \(0\) assumption in order to invoke Theorem \ref{thm:farkas} on dimensional transversality of the moduli space of ILS \(G^{r,r-s-1}_{d,k}(C)\). In all but one case, terminal family (II), we are able to establish the needed cases of Theorem \ref{thm:farkas} in arbitrary characteristic, so that the proofs in \S \ref{sec:terminal_base} go through essentially without change. In the remaining terminal family, we fall back on the fact that this case is covered by Farkas's existence theorem, \cite[Theorem 0.5]{farkas}.

This section is structured as follows: in \S \ref{sec:positive_char_overview}, we reprove the statements of \S \ref{sec:terminal_base}, assuming the validity of auxiliary transversality statements in arbitrary characteristic. We prove the needed statements in \S \ref{sec:expected_dimension}-\S \ref{sec:lls}. While the arguments of \S \ref{sec:expected_dimension}-\S \ref{sec:lls} may be known to experts, we include them for completeness.

\subsection{Overview}\label{sec:positive_char_overview}

\begin{proposition}\label{prop:serre_Ia_positive}
Proposition \ref{prop:serre_Ia} holds in arbitrary characteristic.
\end{proposition}

\begin{proof}
Recall in this case that \((g,r,d)=(r+1,r,2r)\). In the proof of Proposition \ref{prop:serre_Ia}, we constructed a smooth point \(\uw=(\cL,D,V\subset W)\in G^{r,r-s-1}_{d,k}(C)\) of expected dimension. The hypothesis \(\charac(\bK)=0\) was used to ensure, via Lemma \ref{lem:check_bpf} and Corollary \ref{cor:char_0_serre}, that a neighborhood of \(\uw\) contains bpf points. In Proposition \ref{prop:canonical} below, we show that, for all \(k'\geq k\), the moduli space of ILS \(G^{r,r-s-1}_{d,k'}(C)\) is pure of expected dimension, in arbitrary characteristic. This suffices to apply the argument of Lemma \ref{lem:check_bpf}. 

Indeed, the linear series \((\cL,W)\in G^r_d(C)\) underlying \(\uw\in G^{r,r-s-1}_{d,k}(C)\) must be canonical, and is in particular bpf. Therefore, the proof of Lemma \ref{lem:check_bpf} goes through, with only the twisting operations \((T1)\) needed; the dimensional transversality of the spaces \(G^{r,r-s-1}_{d,k'}(C)\) suffices to deduce that \(G^{r,r-s-1}_{d,k}(C)\) is pure of expected dimension, with dense bpf locus.
\end{proof}

\begin{proposition}\label{prop:serre_III_positive}
Proposition \ref{prop:serre_III} holds in arbitrary characteristic.
\end{proposition}

\begin{proof}
Recall in this case that \(\rho'=0\) and \(\rho\leq r-s-1\). In the proof of Proposition \ref{prop:serre_III}, we first used the hypothesis \(\charac(\bK)=0\) to invoke Bertini smoothness for the map \(f\colon C\to \bP^{\rtilde}\). In arbitrary characteristic, Bertini smoothness still applies as long as \(f\) is not everywhere ramified, which holds in our case by Proposition \ref{prop:rho_0_unramified} below.

We again used the hypothesis \(\charac(\bK)=0\) to ensure that a neighborhood of \(\uw\) contains bpf points. In fact, in arbitrary characteristic, it is true that \emph{any} point \(\uw\in G^{r,r-s-1}_{d,k}(C)\) is bpf. Indeed, we have
\[
\rho(g,r-s-1,d-k-1)<\rho'=0,
\]
so \(V\in G^{r-s-1}_{d-k}(C)\) must be bpf. Similarly, we have
\[
\rho(g,r,d-1)=\rho-(r+1)<0,
\]
by Lemma \ref{lem:rho_prime_0_profile}, so \(W\in G^r_d(C)\) is also bpf.
\end{proof}

\begin{proposition}\label{prop:serre_Ib_positive}
Proposition \ref{prop:serre_Ib} holds in arbitrary characteristic.
\end{proposition}

\begin{proof}
In the proof of Proposition \ref{prop:serre_Ib}, we applied Theorem \ref{thm:farkas} to the Serre dual space \(G^{\rtilde,\rtilde-\ws-1}_{\dtilde,k}(C)=G^{\rtilde,1}_{\dtilde,k}(C)\) to argue that any point is complete. In arbitrary characteristic, we may instead apply Proposition \ref{prop:hurwitz} below.

Let \(\uw\in G^{r,r-s-1}_{d,k}(C)\) be the Serre dual ILS, which is also complete. We must now show that some open neighborhood of \(\uw\) contains bpf points. Let \(\uw'=(\cL,D',V'\subset W)\in G^{r,r-s-1}_{d,k}(C)\) be a general point of any irreducible open neighborhood of \(\uw\). Note in particular that, because \(\rho=0\), the underlying linear series \((\cL,W)\in G^r_d(C)\) must be the same as that of \(\uw\). In particular, \((\cL,W)\) is bpf. Furthermore, \(\uw'\) is complete, because completeness is open. 

Therefore, if \(\uw'\) is not bpf, then we may apply operation (T1) (Lemma \ref{lem:check_bpf}) at least once, to produce a \(\rho_{\bullet}\)-dimensional family of complete ILS in \(G^{r,r-s-1}_{d,k+1}(C)\). By Serre duality (Lemma \ref{lem:serre}), we also obtain a \(\rho_{\bullet}\)-dimensional family of complete ILS in \(G^{\rtilde+1,\rtilde-\ws}_{\dtilde,k+1}(C)=G^{\rtilde+1,1}_{\dtilde,k+1}(C)\), which contradicts Proposition \ref{prop:hurwitz}. We conclude that \(\uw'\) is bpf, as needed.
\end{proof}

\begin{proposition}\label{prop:serre_II_positive}
Let \((g,d,k;r,s)\) be a terminal tuple of type (II).
\begin{enumerate}
\item[(a)] If \(k=2s+2\) (the largest possible), then Proposition \ref{prop:serre_II} holds in arbitrary characteristic.
\item[(b)] If \(k\leq 2s+1\), then \(\Ex(g,d,k;r,s)\) holds \emph{unconditionally}, in arbitrary characteristic.
\end{enumerate}
\end{proposition}

\begin{proof}
\,
\begin{enumerate}
\item[(a)] 
Following the proof of Proposition \ref{prop:serre_II}, we must first show that any smooth point \(\underline{\widetilde{W}}=(\cM,D,\widetilde{V}\subset \widetilde{W})\in G^{\rtilde,\rtilde-\ws-1}_{\dtilde,k}(C)=G^{\rtilde,1}_{\dtilde,k}\) of dimension 0, the expected, is complete. Observe first that, by Lemma \ref{lem:rho_bullet_0_profile}, because \(k=2s+2\), we have \(\rho'=0\). Therefore, it is impossible to enlarge the linear series \(\widetilde{W}\). On the other hand, if it is possible to enlarge the linear series \(\widetilde{V}\subset H^0(C,\cM(-D))\), then \(\underline{\widetilde{W}}\in G^{\rtilde,\rtilde-\ws-1}_{\dtilde,k}(C)\) moves in a positive dimensional family, a contradiction.

Therefore, Lemma \ref{lem:serre} applies, yielding a smooth point \(\uw=(\cL,D,V\subset W)\in G^{r,r-s-1}_{d,k}(C)\) of dimension 0. It suffices to show that \(\uw\) is bpf. \(V\) must be bpf because \(\rho'=0\). Therefore, if \(\uw\) is not bpf, then there must exist a point \(p\in D\) for which
\[
h^0(C,\cL(-p))=h^0(C,\cL)=r+1.
\]

In particular, for any \(q\in C\), we also have \(h^0(C,\cL(q-p))\geq r+1\). Because equality hold at $q=p$, by semicontinuity, equality holds for a general \(q\in C\). Now, we obtain a one-parameter family of ILS 
\[
\uw(q):=(\cL(q-p),D+q-p,V\subset H^0(C,\cL(q-p)))\in G^{r,r-s-1}_{d,k}(C),
\] 
containing \(\uw(p)=\uw\). However, \(\uw\in G^{r,r-s-1}_{d,k}(C)\) is a smooth point of dimension \(0\) which is a contradiction.

\item[(b)] This case is covered by Farkas's existence theorem, \cite[Theorem 0.5(i)]{farkas}. In particular, the inequality \(2f\leq e-1\) translates, after substituting \(e=k\) and \(f=k-s-1\), precisely into \(k\leq 2s+1\). 

Now, \cite[Theorem 0.5(i)]{farkas} asserts the existence of an ILS with invariants \((g,d,k;r,s)\). In fact, the (limit) ILS constructed in its proof is both smooth of expected dimension and bpf. We verify these statements in Corollary \ref{cor:farkas_nice} below. In contrast to the proof of Theorem \ref{thm:farkas}, the proof of \cite[Theorem 0.5]{farkas} is independent of characteristic.
\end{enumerate}
\end{proof}

\begin{proposition}\label{prop:all_zeros_positive}
Proposition \ref{prop:all_zeros} holds in arbitrary characteristic.
\end{proposition}

\begin{proof}
In the proof of Proposition \ref{prop:all_zeros}, we applied Theorem \ref{thm:farkas} to \(G^{r,1}_{3r,2r-2}(C)\) in order to ensure that the constructed ILS \(\uw\) is smooth of dimension \(0\), the expected. In arbitrary characteristic, this can be replaced by Proposition \ref{prop:hurwitz}, taking \(k=2r-2\). Then, the fact that the constructed point \(\uw\in G^{r,1}_{3r,2r-2}(C)\) is bpf follows in arbitrary characteristic immediately from the fact that \(\rho=\rho'=0\).

Alternatively, this case is covered by \cite[Theorem 0.5(ii)]{farkas}, along with Corollary \ref{cor:farkas_nice}.
\end{proof}

From here, we complete the proof of Theorem \ref{thm:main}, exactly as we did in characteristic \(0\), except that Proposition \ref{prop:serre_II_positive} is now strictly stronger than Proposition \ref{prop:serre_II}.

\subsection{Dimensional transversality}\label{sec:expected_dimension}

In this section, we prove special cases of Theorem \ref{thm:farkas} in arbitrary characteristic. We work only in the generality needed to complete the proof of Theorem \ref{thm:main} in arbitrary characteristic, as explained in the previous section.

\begin{proposition}\label{prop:canonical}
Fix \(r,s\) with \(s \leq r-2\), and \((g,d,k)=(r+1,2r,k)\), where \(k \geq s+2\), cf. Definition \ref{def:terminal}(Ia). Let \(C\) be a general curve of genus \(g=r+1\). Then, the moduli space of ILS \(G^{r,r-s-1}_{d,k}(C)\) is pure of dimension
\[
\rho_{\bullet}=(r-s)(s+1)-(r-s-1)k,
\]
the expected, and is empty if \(\rho_{\bullet}<0\).
\end{proposition}

\begin{proof}
Let \(\uw=(\cL,D,V \subset W)\in G^{r,r-s-1}_{d,k}(C)\) be a point. Then, \((\cL,W)\in G^r_d(C)\) must be the canonical linear series. Therefore, the space of ILS is identified with the space of choices of:
\begin{itemize}
\item a linear series \((\cN,V)\in G^{r-s-1}_{d-k}(C)\), and
\item an effective divisor \(D\in |K_C\otimes \cN^\vee|\).
\end{itemize}

The space \(G^{r,r-s-1}_{d,k}(C)\) is further stratified by the value of \(h^0(\cN)\); write \(G_\delta \subset G^{r,r-s-1}_{d,k}(C)\) for the locally closed subset where \(h^0(\cN)=r-s+\delta\), where \(\delta\geq 0\). Because
\[
k\geq s+2 \implies g-(d-k)+(r-s-1)<0,
\]
\(\cN\) lies in the Brill-Noether-special range (see \cite[Lemma IV.3.5]{acgh}). Thus, the dimension of the space \(W^{(r-s-1)+\delta}_{d-k}(C)\) of line bundles \(\cN\) underlying \(G_\delta\) is
\[
\rho(g,(r-s-1)+\delta,d-k)=\rho' + \delta\bigl((d-k)-g-2(r-s-1)-1\bigr)-\delta^2.
\]
For a fixed \(\cN\in W^{(r-s-1)+\delta}_{d-k}(C)\), the space of possible linear series \(V\) has dimension
\[
\dim\bigl(\Gr(r-s,r-s+\delta)\bigr)=(r-s)\delta,
\]
and the dimension of the space of possible \(D\in |K_C\otimes \cN^\vee|\) is
\[
h^0(K_C\otimes \cN^\vee)-1=k-s+\delta-1.
\]

Therefore, we have
\begin{align*}
\dim\bigl(G_\delta\bigr)&=\bigl(\rho'+k-s-1\bigr)+\delta\bigl((d-k)-g-(r-s)+2\bigr)-\delta^2 \\
&= \bigl(\rho'+k-s-1\bigr)+\delta\bigl(s+1-k\bigr)-\delta^2
\end{align*}
which is a strictly decreasing function in \(\delta\) because $k \geq s+2$. It follows that
\[
\dim\bigl(G^{r,r-s-1}_{d,k}(C)\bigr)=\dim\bigl(G_0\bigr)=\rho'+k-s-1=\rho_{\bullet},
\]
cf. Lemma \ref{lem:rho_0_profile}(a).
\end{proof}

\begin{proposition}\label{prop:hurwitz}
Fix \(r,s,k\) with \(s\leq r-2\) and \(k\geq s+2\). Set
\[
(g,\dtilde,k;\rtilde,\ws)=(2r+2,r+2+k,k;k-s,k-s-2),
\]
cf. Proposition \ref{prop:serre_Ib}. Let \(C\) be a general curve of genus \(g\). Then, the moduli space of ILS \(G^{\rtilde,\rtilde-\ws-1}_{\dtilde,k}(C)=G^{\rtilde,1}_{\dtilde,k}(C)\) is pure of dimension
\[
\rho_{\bullet}=(r-s)(s+1)-(r-s-1)k,
\]
the expected, and is empty if \(\rho_{\bullet}<0\).
\end{proposition}

\begin{proof}
Let \(\uw=(\cL,D,V\subset W)\in G^{\rtilde,1}_{\dtilde,k}(C)\) be a general point. Then, \((\cN,V)\in G^1_{r+2}(C)\) is one of finitely many possible points, underlying a morphism \(f\colon C\to \bP^1\) of degree \(r+2\). To compute the dimension of \(G^{\rtilde,1}_{\dtilde,k}(C)\) near \(\uw\), it suffices to fix \(f\). Then, an ILS \(\uw\) is determined by the data of:
\begin{itemize}
\item a line bundle \(\cL\in W^{\rtilde}_{\dtilde}(C)\), such that \(\cL\otimes \cN^\vee=\cL\otimes f^*\cO_{\bP^1}(-1)\) is effective,
\item an effective divisor \(D\in |\cL\otimes \cN^\vee|\), and
\item a linear series \(W\subset H^0(C,\cL)\) of rank \(\rtilde+1\), containing \(V\).
\end{itemize}

To understand the space of possible \(\cL\in \Pic^{\dtilde}(C)\) satisfying the first condition, we will apply the main dimensional transversality theorem of Hurwitz-Brill-Noether theory, \cite[Theorem 1.2]{hlar}. Namely, consider the vector bundle \(f_*\cL\) on \(\bP^1\), which has rank \(\dtilde-k=r+2\) and degree
\[
\dtilde-g+1-(r+2)=k-2r-1.
\]
Then, we have a splitting
\[
f_*\cL\cong \cO(e_1)\oplus \cdots \oplus \cO(e_{r+2}),
\]
where \(e_1\leq \cdots \leq e_{r+2}\). 

The Picard variety is stratified by loci \(\Sigma_{\vec e}\subset \Pic^{\dtilde}(C)\) for which \(f_*\cL\) has splitting type \(\vec e=(e_1,\ldots,e_{r+2})\). \cite[Theorem 1.2]{hlar} asserts that \(\Sigma_{\vec e}\) is non-empty and pure of dimension
\[
g-\sum_{1\leq i<j\leq r+2}(e_j-e_i-1),
\]
if this quantity is non-negative, and empty otherwise. In terms of the splitting of \(f_*\cL\), the condition that \(h^0(\cL)\geq \rtilde+1\) is equivalent to
\[
\sum_{i=1}^{r+2}\max\{0,e_i+1\}\geq \rtilde+1,
\]
and the condition that \(\cL\otimes f^*\cO_{\bP^1}(-1)\) is effective is equivalent to simply \(e_{r+2}\geq 1\), by the projection formula.

Now, the dimension of \(G^{\rtilde,1}_{\dtilde,k}(C)\) near \(\uw\) is equal to
\[
\begin{aligned}
\Delta_{\vec e}
&:= \dim(\Sigma_{\vec e})+\bigl(h^0(C,\cL\otimes\cN^\vee)-1\bigr)+\dim\bigl(\Gr(\rtilde-1,H^0(C,\cL)/V)\bigr)\\
&= \dim(\Sigma_{\vec e})+\bigl(h^0(\bP^1,f_*\cL\otimes \cO_{\bP^1}(-1))-1\bigr)+\dim\bigl(\Gr(\rtilde-1,h^0(\bP^1,f_*\cL)-2)\bigr).
\end{aligned}
\]
All three quantities are in particular determined by the splitting type \(\vec e\). It is now a purely numerical statement that \(\Delta_{\vec e}\leq \rho_{\bullet}\), with equality if and only if
\[
\vec e=(e_1,\ldots,e_{r+s-k+2},0,\ldots,0,1),
\]
where \(e_{r+s-k+2}-e_1\leq 1\). While one can prove this numerical statement directly, we have the following cheaper finish.

Namely, because the non-emptiness and dimensional transversality of the splitting locus \(\Sigma_{\vec e}\subset \Pic^{\dtilde}(C)\) are independent of characteristic, the dimension of \(G^{\rtilde,1}_{\dtilde,k}(C))\) is also independent of characteristic. Therefore, we may conclude immediately by the characteristic \(0\) statement, Theorem \ref{thm:farkas}.
\end{proof}

\begin{remark}\label{rem:hurwitz_characteristic}
As explained in \cite{hlar}, the main theorems of Hurwitz-Brill-Noether theory hold for a general point \(f\colon C\to \bP^1\) in the Hurwitz space of simply branched covers, lying in the same component as a particular degenerate cover \(f\). In our case, \(f\) may be regarded as a general point of the universal Brill-Noether variety \(W^1_{r+2}\) over the moduli space of smooth curves, which is irreducible in any characteristic \cite[\S 10]{llv}. Alternatively, the irreducibility of simply branched Hurwitz spaces is now known in arbitrary characteristic \cite{cht2}.
\end{remark}

\subsection{Transversality for limit linear series}\label{sec:lls}

Let $(C,p_0) \in \cM_{g,1}$ be a general pointed curve. In this section, we study, in arbitrary characteristic, linear series on $(C,p_0)$ with pointed Brill-Noether number $0$, along with their limits on chains of elliptic curves. Such objects underlie the construction of degenerate ILS in the proof of \cite[Theorem 0.5]{farkas}, which we verify smooth to nice maps in Corollary \ref{cor:farkas_nice}. We assume a basic familiarity with the theory of limit linear series \cite{eh_lls}.

Fix a rank $r\ge1$, a degree $d$, and a vanishing sequence $0 \leq a_0 < \cdots < a_r \leq d$, for which
\[
\rho(g,r,d,\vec{a}) := \rho(g,r,d) - \sum_{i=0}^r (a_i-i) = 0.
\]
Consider a one-parameter degeneration of $(C,p_0)$ to a nodal curve $(C_0,p_0)$, where $C_0 = E_1 \cup \cdots \cup E_g$ is a chain of elliptic curves with $p_0 \in E_1$, with smooth total space. For $i = 1,\ldots,g-1$, write $p_i = E_i \cap E_{i+1}$, and fix also a smooth point $p_g \in E_g$. Assume that, for all $i = 1,\ldots,g$, the divisor $p_{i-1}-p_i$ is not torsion on $E_i$. The set of limit linear series of rank $r$ and degree $d$ on $C_0$ with vanishing sequence $\vec{a}$ at $p_0$ admits an explicit combinatorial description; see, e.g., \cite[\S 1]{lt}. We review the important properties.

Let $W_0$ be such a limit linear series, and let $(\cL_i,W_i)$ be the aspect of $W_0$ on $E_i$. Let
\[
a_0^i < \cdots < a_r^i \quad \text{and} \quad b_0^i > \cdots > b_r^i
\]
be the vanishing sequences of $(\cL_i,W_i)$ at $p_{i-1}$ and $p_i$, respectively. For $i=1$, we have $a^1_j = a_j$ for all $j$. For $i$ arbitrary, there is a unique index $t_i \in \{0,\ldots,r\}$ for which $a_{t_i}^i + b_{t_i}^i = d$. For all $j \neq t_i$, we have instead $a_j^i + b_j^i = d-1$. From here, we have necessarily $\cL_i \cong \cO_{E_i}(a_{t_i}^i p_{i-1} + b_{t_i}^i p_i)$. A basis of $W_i$ is given by the unique, up to scaling, nonzero sections $\sigma^i_j\in H^0(\cL_i(-a_j^i p_{i-1} - b_j^i p_i))$, for $j = 0,\ldots,r$. Because $p_{i-1}-p_i$ is assumed not to be a torsion divisor on $E_i$, the remaining vanishing point of $\sigma^i_j$, for $j \neq t_i$, is not equal to $p_{i-1}$ or $p_i$.

The collection of integers $t_1,\ldots,t_g$, along with the fact that $W_0$ is refined ($b^{i-1}_j+a^i_j=d$ for all $i,j$), determines the limit linear series $W_0$. There is an additional combinatorial condition on the $t_i$, which we will not need to recall. Each of the limit linear series described above defines a smooth point of dimension $0$ on the moduli space of limit linear series on $(C_0,p_0)$, and smooths to a linear series on $(C,p_0)$. In particular, the moduli space of limit linear series on $(C,p_0)$ with vanishing sequence $\vec{a}$ is reduced of dimension $0$; see \cite{osserman_connected} for a generalization.

\begin{proposition}\label{prop:rho_0_unramified}
Let \(g,r,d\) be integers with \(\rho(g,r,d)=0\). Let \(C\) be a general curve of genus \(g\), and let \(f\colon C\to \bP^r\) be the non-degenerate map underlying a point of \(G^r_d(C)\). Then, \(f\) is generically unramified.
\end{proposition}

\begin{proof}
It suffices to show that, for each of the limit linear series $W_0$ on $C_0$ described above, not every point of $C_0$ is a ramification point of $W_0$. This is immediate from the fact that, for any $i$ and any $j \neq t_i$, the section $\sigma^i_j \in W_i$ has a unique point other than $p_{i-1}$ and $p_i$ of simple vanishing on $E_i$.
\end{proof}

\begin{proposition}\label{prop:pointed_rho_zero}
Let $(C,p_0)$ be a general pointed curve of genus $g$. Fix $r,d$, and a vanishing sequence $\vec{a}$ such that $\rho(g,r,d,\vec{a})=0$. Let $W$ be a linear series of rank $r$, degree $d$, and vanishing sequence $\vec{a}$ at $p_0$.

Fix an integer $s$ with $-1 \leq s \leq r-2$, and let $V \subset W$ be the unique subspace of sections of dimension $r-s$ with vanishing sequence $a_{s+1}<\cdots<a_r$ at $p_0$. Then, $V$ has no base points on $C$ away from $p_0$.
\end{proposition}

Note that the case $s=-1$ amounts to the statement that has no base points on $C$ away from $p_0$.

\begin{proof}
Fix a degeneration
\[
(C,p_0) \rightsquigarrow (C_0 = E_1 \cup \cdots \cup E_g,p_0)
\]
as above, and let $W_0$ be the limit of $W$ on $C_0$. Let $V_0$ be the limit of the subspace $V \subset W$ on $C_0$, and write $V_1,\ldots,V_g$, with $V_i \subset W_i$, for the aspect of $V_0$ on $E_i$. We claim that, for all $i = 1,\ldots,g$, we have
\[
V_i = \langle \sigma^i_{s+1},\ldots,\sigma^i_r\rangle.
\]
Equivalently, $V_i \subset W_i$ is the unique subspace with vanishing sequences
\[
a^i_{s+1}<\cdots<a^i_r \quad \text{and} \quad b^i_{s+1}>\cdots>b^i_r
\]
at $p_{i-1}$ and $p_i$, respectively. 

We argue by induction on $i$. This is clear for $i = 1$, by the requirement that $V$ have vanishing sequence $a_{s+1}<\cdots<a_r$ at $p_0$. Then, if $V_i = \langle \sigma^i_{s+1},\ldots,\sigma^i_r\rangle$, then in order for $V_i,V_{i+1}$ to satisfy the compatibility condition for limit linear series, the vanishing sequence of $V_{i+1}$ at $p_i$ must be \emph{at least}
\[
(d-b^i_{s+1},\ldots,d-b^i_r) = (a^{i+1}_{s+1},\ldots,a^{i+1}_r).
\]
It follows that we must have $V_{i+1} = \langle \sigma^{i+1}_{s+1},\ldots,\sigma^{i+1}_r\rangle$, as needed.

Now, it suffices to show that each $V_i$ has no base points away from $p_{i-1}$ and $p_i$.\footnote{If so, a base point of $V$ could only specialize to a node of $C_0$, but this is also not possible by standard semi-stable reduction argument.}

Recall that each section $\sigma^i_j$ has at most one point $q^i_j$ of simple vanishing away from $p_{i-1}$ and $p_i$. Moreover, varying over all $j$, the points $q^i_j$ must be pairwise distinct, or else $p_{i-1}-p_i$ is torsion. Because $\dim(V_i)=r-s \geq 2$, the base-point-freeness follows.
\end{proof}

\begin{corollary}\label{cor:farkas_nice}
Let $(g,d,k;r,s)$ be a terminal tuple of type (II), with $k \leq 2s+1$. Let $\underline{W_0}$ be the limit ILS on the singular curve $C_0 = Y \cup Z$ of genus $g$ constructed in \cite[\S 3]{farkas}. Then, $\underline{W_0}$ smooths to an ILS $\uw \in G^{r,r-s-1}_{d,k}(C)$ on a smooth curve $C$, which is a smooth point of dimension $0$, the expected, and bpf.
\end{corollary}

\begin{proof}
We first recall the needed aspects of Farkas's construction. The curve $C_0$ is given by the nodal union of smooth components $Y,Z$ of genus $k$ and $g-k$, respectively, intersecting at $p$. The pointed curves $(Y,p)$ and $(Z,p)$ are assumed to be general. The limit ILS $\underline{W_0}$ is built from three pieces:
\begin{itemize}
\item a linear series $(\mathcal{A},W_{\mathcal{A}}) \in G^{r-s-1}_{d-k}(Y)$, with \emph{ramification}\footnote{While the results above involving limit linear series on chains of elliptic curves were more easily stated in terms of \emph{vanishing} sequences $a_i=\alpha_i+i$, here we match Farkas's convention to index instead the ramification sequences.} sequence $\alpha=(\alpha_0\leq\cdots\leq\alpha_{r-s-1})$ at $p$, where $\alpha_{r-s-1}-\alpha_0 \leq 1$
\item a linear series $(\cL,W_{\cL}) \in G^s_{k+s}(Y)$, with ramification sequence $\beta=(  \beta_0\leq\cdots\leq\beta_s)$ at $p$, where $\beta_s - \beta_0 \leq 1$, and
\item a linear series $\mathfrak{l}_Z \in G^r_d(Z)$, with ramification sequence
\begin{equation}\label{eqn:ram-seq-l}
\gamma: \quad k-\beta_s \leq \cdots \leq k-\beta_0 \leq d-r-\alpha_{r-s-1} \leq \cdots \leq d-r-\alpha_0
\end{equation}
at $p$.
\end{itemize}

The integers $\alpha_i$ and $\beta_i$ are chosen in the unique possible way such that each of the linear series $W_{\mathcal{A}}$, $W_{\cL}$, and $\mathfrak{l}_Z$, together with their ramification data at $p$, defines a smooth, $0$-dimensional point in its respective moduli space of pointed linear series.

Farkas shows that, under the assumption $k \leq 2s+1$, \footnote{The assumption $k \leq 2s+1$ here is equivalent to condition (i) ($2f \leq e-1$) in \cite[Theorem 0.5]{farkas}.} one can furthermore arrange for
\begin{equation}\label{eqn: H0-vanishing-Gabi}
H^0(Y,\cL\otimes \mathcal{A}^{\vee}\otimes \cO_Y((d-k-s-1)p))=0.
\end{equation}
Therefore, by Riemann-Roch, the degree $k=g(Y)$ line bundle $\cL\otimes \mathcal{A}^{\vee}\otimes \cO_Y((d-k-s)p)$ on $Y$ has a \emph{unique} effective divisor $D$ in its linear system, disjoint from $p$.

From here, $(\mathcal{A}(D), W_{\mathcal{A}}(D) \oplus W_{\cL}((d-k-s)p))$\footnote{One needs that $d \geq k+s$, which holds in terminal family (II). } is a rank $r$, degree $d$ linear series on $Y$, and together with $\mathfrak{l}_Z$, gives a refined limit linear series $W_0$ on $C_0$. Now, write $V_Z$ for the $(r-s)$-dimensional subspace of sections of $\mathfrak{l}_Z$ with ramification sequence
\[
d-r-\alpha_{r-s-1} \leq \cdots \leq d-r-\alpha_0.
\]
Then, $(\mathcal{A},W_{\mathcal{A}})$ and $V_Z(-kp)$ together define a refined sub-limit linear series $V_0 \subset W_0$ of rank $r-s-1$ and degree $d-k$, vanishing along $D \subset Y$. 

We claim that the data of $V_0,W_0,D$ define an isolated reduced point of the moduli space of limit ILS. To see this, consider a first order deformation $(D', \mathcal{M}, V_0' \subseteq W_0)$ of $(D, V_0 \subseteq W_0)$ over $S= \mathrm{Spec}( k[\epsilon]/(\epsilon^2))$. The curve $Y \cup_p Z$ is constant over $S$, and $V_0', W_0'$ are both refined limit linear series. The $Y$-aspect of $V_0'$ and the $Z$-aspect of $W_0'$ cannot deform, because $G^{r-s-1}_{d-k}(Y,\alpha)$ and $G_d^r(Z,\gamma)$ are 0-dimensional and reduced. By semicontinuity, the evaluation map
\[
\bigg( Y-\text{aspect of } W_0' \bigg) \to H^0((d-k-s)p,  \mathcal{M}|_{(d-k-s)p})
\]
has rank exactly $r-s$, so its kernel $K_S$ is locally free of rank $s+1$. Since $G^{s}_{k+s}(Y,\beta)$ is reduced and $0$-dimensional, the $S$-point $K_S(-(d-k-s)p) \in G^{s}_{k+s}(Y,\beta) (S) $ is constant and equal to $W_{\mathcal{L}}$. 
Also, by semicontinuity, \eqref{eqn: H0-vanishing-Gabi} holds over $S$, so $D$ does not move. Because $W_{\mathcal{A}}(D) $ and $W_{\mathcal{L}}((d-k-s)p) $ are constant over $S$, the $Y$-aspect of $W_0'$ is constant. Finally, the $Z$-aspect of $V_0'$ coincides with the intersection of the $Z$-aspect of $W_0'$ vanishing to order at least $d-r-\alpha_{r-s-1}+s+1$ at $p$, which is also fixed.

Finally, it suffices to check that the respective aspects of $V_0,W_0$ have no base points away from $p$. In the case of $V_0$, this follows from Proposition \ref{prop:pointed_rho_zero}, applied to $(\mathcal{A},W_{\mathcal{A}})$ and $V_Z \subset \mathfrak{l}_Z$. In the case of $W_0$, it again follows from Proposition \ref{prop:pointed_rho_zero}, applied to $(\cL,W_{\cL})$ and all of $\mathfrak{l}_Z$.
\end{proof}

\bibliographystyle{alpha} 
\bibliography{existence_secant.bib}

\end{document}